\documentclass[dvipsnames]{amsart}
\usepackage{pict2e}
\usepackage{amsmath,amssymb,amsthm}
\usepackage{fullpage}
\usepackage{todonotes}

\allowdisplaybreaks
\usepackage{soul}

\usepackage{color}
\usepackage{hyperref}
\hypersetup{
	colorlinks=true,
	linkcolor=blue,
	citecolor=blue,
	urlcolor=blue
}

\usepackage{xcolor}

\usepackage{tikz}

\newtheorem{theorem}{Theorem}[section]
\newtheorem{lemma}[theorem]{Lemma}	
\newtheorem{proposition}[theorem]{Proposition}

\theoremstyle{definition}
\newtheorem{definition}[theorem]{Definition} 
\newtheorem{remark}[theorem]{Remark}

\newcommand{\longdownarrow}{\lower 1.4ex\hbox{\begin{picture}(18,18)(0,0)
\thicklines
\put(0,18){\vector(0,-1){18}}
\end{picture}}}
\newcommand{\longsearrow}{\lower 1.4ex\hbox{\begin{picture}(18,18)(0,0)
\thicklines
\put(0,18){\vector(1,-1){18}}
\end{picture}}}
\newcommand{\longssearrow}{\lower 1.4ex\hbox{\begin{picture}(18,18)(0,0)
\thicklines
\put(0,18){\vector(1,-2){18}}
\end{picture}}}
\newcommand{\longeearrow}{\lower 1.4ex\hbox{\begin{picture}(0,0)(9,9)
\thicklines
\put(0,18){\vector(1,0){18}}
\end{picture}}}
 \usepackage{relsize}
\theoremstyle{definition} 
\newtheorem*{ack}{Acknowledgements}

\numberwithin{equation}{section}
\newcommand{\C}{\mathcal{C}}

\newcommand{\N}{\mathbb{N}}
\newcommand{\Z}{\mathbb{Z}}
\title[Spectral complexes from truncated multicomplexes]{Spectral complexes from truncated multicomplexes} 
 
\author[F.~Tripaldi]{Francesca Tripaldi}

\address[F. Tripaldi]%
{Department of Pure Mathematics, University of Leeds, Woodhouse, LS2 9JT Leeds, UK} 
\email{f.tripaldi@leeds.ac.uk}

\date{}

\begin{document}
\maketitle
\begin{abstract}
  This paper introduces a new construction of subcomplexes associated with a truncated multicomplex. Inspired by the machinery of spectral sequences, this construction yields a collection of interrelated subcomplexes whose differentials coincide with the spectral sequence differentials. These complexes refine the Rumin complex and retain the cohomology of the underlying multicomplex, providing a new tool for the study of subRiemannian geometry, particularly on Carnot groups.
\end{abstract}

\section{Introduction}
Extracting meaningful subcomplexes from cochain complexes is a central and very active theme in differential geometry. The underlying goal is to capture the extra algebraic, geometric or analytic structures that may be present in a given setting. Such subcomplexes are typically designed to be more closely adapted to the problem at hand, allowing one to extract sharper topological or metric invariants, or to construct differential operators with stronger analytic properties.

In the case of a complex manifold $X$, the presence of a complex structure endows the de Rham complex $(\Omega^\bullet(X),d)$ with additional algebraic properties. This is due to the fact that the complex structure induces a splitting of the tangent bundle into holomorphic and anti-holomorphic components, and consequently provides a decomposition of the exterior derivative such that $d=\partial+\overline{\partial}$ satisfying $0=d^2=\partial^2=\overline{\partial}^2=\partial\overline{\partial}+\overline{\partial}\partial$. This endows the de Rham complex with the structure of a bicomplex. In particular, since $\overline{\partial}^2=0$, one may consider the Dolbeault complex $(\Omega^\bullet(X),\overline{\partial})$, whose cohomology provides a fundamental analytic invariant of complex 
manifolds, relating holomorphic complexity with topology using the Fr\"olicher 
spectral sequence. Indeed, the complex $(\Omega^\bullet(X),\overline{\partial})$ plays a crucial role in the study and classification of complex structures \cite{huybrechts2005complex}.

When the complex manifold $X$ carries additional geometric structures, further information can be extracted using other subcomplexes and their associated cohomology groups. In the case of a K\"ahler manifold $X$, the complex, symplectic, and Riemannian structures are mutually compatible. This compatibility allows us to detect information about the underlying manifold via the Bott-Chern and Aeppli cohomology groups
\begin{align*}
    H_A^{\bullet,\bullet}:=\frac{\ker\partial\overline{\partial}}{\mathrm{Im}\left(\partial+\overline{{\partial}}\right)}\ \text{ and }\ H_{BC}^{\bullet,\bullet}:=\frac{\ker \left(\partial+\overline{\partial}\right)}{\mathrm{Im}\,\partial\overline{\partial}}\,,
\end{align*}
which are well-defined since $d\circ \partial\overline{\partial}=\partial\overline{\partial}\circ d=0$.
The study of these cohomology groups later inspired the construction of a new subcomplex, known in the literature as either the Bigolin \cite{piovani2025cohomologybigolincomplex}, Schweitzer \cite{StelzigPrimo}, or Aeppli-Bott-Chern complex \cite{Piovani2}, which turns out to be elliptic \cite{Stelzig} and carries Hodge theoretic properties as well as cohomological invariants \cite{piovani2025cohomologybigolincomplex}.

In the case of symplectic manifolds which may not be K\"ahler in general, a symplectic analogue of the complex Bott-Chern and Aeppli cohomology
 groups have also been introduced in \cite{TsengYau}, carrying analogous properties to the K\"ahler ones \cite{TardiniTomassini}. Finally, both the Dolbeault complex \cite{Ciricialmost} and Bott–Chern and Aeppli cohomology \cite{SillariTomassini} have been successfully extended to the broader context of almost complex manifolds.

Another fundamental approach to the construction of subcomplexes is provided by Bernstein-Gelfand-Gelfand (BGG) sequences, which play a central role in the study of parabolic geometries \cite{vcap2024parabolic}. Originally introduced in the representation theory of semisimple Lie algebras 
$\mathfrak{g}$ \cite{bernstein1975differential}, BGG resolutions were later extended to the setting of parabolic subalgebras $\mathfrak{p}\subset\mathfrak{g}$ \cite{lepowsky1977generalization}. These algebraic constructions were soon found to have deep geometric significance. In the specific setting of parabolic subalgebras \cite{lepowsky1977generalization}, the homomorphisms appearing in the BGG resolutions correspond to invariant differential operators acting on sections of homogeneous vector bundles over the generalised flag manifold $G/P$ induced by irreducible representations of an appropriate parabolic subgroup $P\subset G$. For suitable choices of $G$ and $P$, the corresponding $G/P$ coincides with the homogeneous model for some fundamental geometric structures, such as CR and conformal geometry \cite{eastwood1987conformally}. These first works motivated the systematic study of geometries with a homogeneous model, now commonly referred to as parabolic geometries. 
In the early 2000s, a formulation of BGG sequences rooted in differential geometry was developed in \cite{vcap2001bernstein}. In this framework, one constructs sequences of differential operators that are intrinsic to the associated Cartan geometry. On geometries that are locally isomorphic to the homogeneous model $G/P$, the induced BGG sequence is a complex with possibly higher order differential operators that computes the same cohomology as the (twisted) de Rham cohomology.
 We refer to the introduction in \cite{cap2025bgg} for a thorough exploration of the construction and geometric motivation behind BGG sequences.

 Even though the BGG machinery heavily relies on the associated canonical Cartan geometry via tractor bundles and connections, making the original construction not easily accessible, there are specific settings where this can be greatly simplified: on filtered manifolds via osculating groups \cite{dave2022graded} and on smooth Riemannian manifolds \cite{fischer2023subcomplexes,cap2025bgg}.
 
 Crucially, the case of filtered manifolds also covers the setting of equiregular subRiemannian manifolds, where the most used and known subcomplex has been the Rumin complex $(E_0^\bullet,d_c)$. This subcomplex was first introduced by M. Rumin on contact 
manifolds in the early 1990s \cite{Rumin1990,Rumin1994}  and later extended to equiregular filtered manifolds \cite{rumin_grenoble,rumin_palermo}.  
 It is worth noting that, starting from a different perspective, an equivalent construction of $(E_0^\bullet,d_c)$ had already been developed independently. Motivated by
the study of Monge-Amp\`ere equations on contact manifolds, Lychagin constructed an equivalent subcomplex \cite{lychagin1979contact} in 1979 (and later expanded in \cite{lychagin1994non}), intended as a contact variant of the
effective cohomology groups of symplectic manifolds \cite{bouche}. After remaining an open question since the early 2000s, it was recently shown that BGG sequences and the Rumin complex coincide in the setting of homogeneous groups \cite{F+T1} and, more generally, on equiregular filtered manifolds \cite{fischer2023subcomplexes}. 

With the aim of extending both geometric and analytic constructions from the Riemannian to the subRiemannian setting, the Rumin complex has become an essential replacement for the de Rham complex. 
In the case of contact manifolds, the subcomplex $(E_0^\bullet,d_c)$ has been successfully used to construct (maximally) hypoelliptic Hodge-Laplacians on forms \cite{Rumin1994} with applications to computing analytic torsion \cite{Rumin+Seshadri2012,kitaoka2020analytic,kitaoka2022harmonic,rumin2024topological}, sharp quasi-isometry invariants \cite{rumin1999differential,BFP1,pansu2019averages,Pansu+Rumin2018,BFP3,BFP4,baldi2024continuous}, and form-valued partial differential equations \cite{GagliardoNiremb,FranchiMontefalconeSerra,baldi2022sobolev,BaldiRosa}. More recently, further applications have emerged in geometric measure theory, particularly through connections with intrinsic Lipschitz graphs \cite{FranchiSerapioni} and the validity of Stokes’ theorem for $(E_0^\bullet,d_c)$ \cite{franchi2006div,StokesHeis,dimarco2025submanifoldsboundarysubriemannianheisenberg}, laying the groundwork for a theory of subRiemannian currents on contact manifolds \cite{julia2023flat,canarecci2021sub,canarecci2023horizontal,franchi2025currentsheisenberggroups}. Nevertheless, even though the Rumin complex can be defined on arbitrary (equiregular) subRiemannian manifolds, these results do not hold outside of contact and $(2,3,5)$-manifolds (we refer to \cite{haller2024regularizeddeterminantsrumincomplex,haller2025etainvariant235distributions,Haller+2025,baldi2025comparing} for some recent analytic results on $(2,3,5)$-manifolds that are locally diffeomorphic to the Cartan group).

\smallskip
 This overview is by no means intended to provide an exhaustive account of all subcomplexes that have been constructed in the literature. Rather, it highlights some of the tools developed in three fundamental geometric settings (complex, parabolic, and subRiemannian geometry) which have proved particularly successful in applications. Many other constructions and extensions exist in additional contexts, including CR manifolds \cite{CASE2016109,case2018boundary,case2023simple,case2025bigraded}, singular filtered manifolds \cite{calin2012integrability}, and finite element exterior calculus \cite{ComplexesFromComplexes,vcap2023bounded,hu2023distributional,rssurvey2025}. The purpose of this brief survey is instead to illustrate a common guiding principle in the study of differential complexes: the identification of a subcomplex that is optimally adapted to a specific geometric or analytic problem. 
 In the present work, we adopt a different perspective. Rather than selecting a single preferred subcomplex, we propose a framework that consists of a \textit{collection} of several natural subcomplexes arising from the underlying geometric structure.
 
 The main source of inspiration for this approach is the Rumin complex in subRiemannian geometry. More specifically, the construction of the ``spectral complexes'' introduced in this paper is motivated by the issues arising when dealing with $(E_0^\bullet,d_c)$ on arbitrary Carnot groups. These difficulties arise precisely when the Rumin differential $d_c$ decomposes as a  sum of differential operators of distinct homogeneous (Heisenberg) orders. This property can be expressed in an equivalent way by saying that the space of Rumin forms splits into multiple (homogeneous) weights in at least one degree \cite{rumin_palermo,BaldiFranchi+2015+333+362,FT}. This situation is by no means the exception when dealing with Carnot groups, but rather the rule (it is reasonable to assume that up to dimension 9 the only groups where Rumin forms appear in only one weight in each degree are Heisenberg groups and the 5-dimensional Cartan group \cite{hakavuori2022gradings,conti2024ricci}). On the other hand, in the special cases where no such splitting occurs and the Rumin differential has a single homogeneous differential order $j$ at a given degree (so that $d_c=d_c^j$), the Rumin differential coincides with the corresponding differential $\Delta_j$ arising from the spectral sequence associated with the filtration by weights \cite{lerario2023multicomplexes}. These observations highlight a fundamental tension. On the one hand, there is a need for a ``thin'' subcomplex, preferably involving differential operators of homogeneous order, which is well suited for analytic purposes. On the other hand, it is essential to ``keep track'' of all Rumin forms because the cohomology of the associated spectral sequence is distributed across all weights. The spectral complexes constructed in this paper are designed to reconcile these competing requirements.
 \smallskip

Inspired by the machinery of spectral sequences, rather than engineering a single optimal subcomplex, we choose to consider a collection of them. Unlike the classical setting of spectral sequences, where at each page we have a family of subcomplexes that run parallel to each other, these new \textit{spectral complexes} do not. Instead, they may intersect and overlap in nontrivial ways. This phenomenon already appears in simple examples. For instance, as shown in Subsection \ref{engel group}, in the case of the Engel group the construction yields two spectral complexes of the form
\begin{align*}
    E_0^0\xrightarrow[]{d_c^1}&\ \quad \ \ E_0^1\ \ \ \ \xrightarrow[]{d_c^2}E_0^2\cap\C_{3,-1}\xrightarrow[]{d_c^3}E_0^3\cap\left(\mathrm{Im}\,d_c^2\right)^\perp\xrightarrow[]{d_c^1} E_0^4\\
    E_0^0\xrightarrow[]{d_c^1}&E_0^1\cap\ker d_c^2\xrightarrow[]{d_c^3} E_0^2\cap\C_{4,-2}\xrightarrow[]{d_c^2} \qquad\ E_0^3\qquad\ \xrightarrow[]{d_c^1} E_0^4
\end{align*}
Clearly, in degrees 0 and 4, the spaces of forms are the same in both subcomplexes, and in degrees 1 and 3 there is an overlap, since $E_0^1\cap\ker d_c^2\subset E_0^1$ and $E_0^3\cap\left(\mathrm{Im}\,d_c^2\right)^\perp\subset E_0^3$. Only in degree 2 do the two subcomplexes fail to intersect. We further stress that these two subcomplexes are equally important and should be considered together: it is the collection of them that carries the information we are looking for.

 Since the construction only relies on the fact that the multicomplex is truncated (and that we have a scalar product for which the Laplacian $\Box_0$ associated with $d_0$ admits a Hodge decomposition, see Remark \ref{d0 closed range}), the subcomplexes will be presented purely in terms of multicomplexes. Consequently, even though the motivating problems arise primarily in the study of Carnot groups, the construction applies more broadly, including to homogeneous groups and equiregular filtered manifolds.

 \smallskip

The paper is organised as follows.
In Section \ref{section Rumin}, after recalling the main definitions and properties of a truncated multicomplex $\C$, we introduce the Rumin complex $(E_0^\bullet,d_c)$ associated with it. Interestingly, from considerations in homotopy theory, we have that every multicomplex $\C$ can be decomposed into a direct sum $K\oplus H$ where $K$ is trivial and $H$ is minimal \cite{dotsenko2015rham}. Under a choice of scalar product satisfying the assumptions of Remark \ref{d0 closed range}, the Rumin complex coincides with this minimal multicomplex $H$. In Section \ref{section 3}, we are able to express the quotients $E_r^{\bullet,\bullet}$ and the differentials $\Delta_r$ of the spectral sequence associated with $\C$ in terms of Rumin forms and their differentials. This step relies on the work of \cite{livernet2020spectral} where the quotients $E_r^{\bullet,\bullet}$ are expressed in terms of certain graded modules $Z_r^{\bullet,\bullet}$ and $B_r^{\bullet,\bullet}$. In Section \ref{section 4}, we introduce the spectral complexes and establish their main properties within this framework. 

These subcomplexes resolve the issues that originally motivated this work: they capture the cohomology of the multicomplex, for every nontrivial space of Rumin forms there exists at least one subcomplex passing through it, the differentials coincide with the spectral sequence differentials $\Delta_r$
 and are therefore independent of the choice of scalar product, and each subcomplex is “thinner”, in the sense that it involves fewer differential orders. As explained in Subsection \ref{considerations}, when Rumin forms occur in more than three weights in some degree, the differentials of the associated spectral complexes may not have homogeneous differential operators.  Understanding the analytic consequences of this phenomenon, and whether further refinements of the construction are required, remains an interesting direction for future research.
\begin{ack}
    This work benefited from numerous conversations over several years. The author is especially grateful to Atabey Kaygun, Filippa Lo Biundo, and Sarah Whitehouse, for their time, interest, and insightful discussions that helped shape this paper.
\end{ack}

\section{The Rumin complex associated with a truncated multicomplex}\label{section Rumin}

\begin{definition}[Definition 2.1 in \cite{livernet2020spectral}]\label{def multicomplex} Let  $k$ be a commutative unital ground ring. An $s$-multicomplex (also known as twisted chain complex) is a $(\Z,\Z)$-graded $k$-module $\C$ equipped with maps $d_i\colon\C\to\C$ for $i\ge 0$ of bidegree $\vert d_i\vert=(i,1-i)$ such that
\begin{align}\label{diff relations}
    \sum_{i+j=n}d_id_j=0\ \text{ for all }n\ge 0\ \text{ and }d_k=0\ \text{ for all }k\ge s\,.
\end{align}

For $\C$ a multicomplex and $(a,b)\in\Z\times\Z$, we write $\C_{a,b}$ for the $k$-module in bidegree $(a,b)$. Inspired by the terminology used to study the de Rham complex on Carnot groups, given an element $x\in\C_{a,b}$, we will refer to $a$ as the \textit{weight} of $x$ and to $a+b$ as the \textit{degree} of $x$. As a shorthand notation, we will write
\begin{align*}
  \text{if }\,  x\in\C_{a,b}\ \text{ then }\ w(x)=a\ \text{ and }\ deg(x)=a+b\,.
\end{align*}
\end{definition}

In order for the constructions presented throughout the paper to make sense, we require $\C$ to be a truncated multicomplex, i.e. an $s$-multicomplex with $s<\infty$, and we are not able to cover the case of $s=\infty$.

\begin{remark}\label{remark 1.2}
    In Definition \ref{def multicomplex}, we chose a particular sign and degree convention. 
    Since the motivational example case behind this work is the de Rham complex, the structure maps $d_i$ are taken with a co-homological convention. What is usually found in the literature is instead first quadrant multicomplexes with differentials $d_i : \C \to \C$ of bidegree $\vert d_i\vert=(-i,i-1)$, more commonly studied in homology theory. Moreover, another common alternative is to require the structure maps to satisfy the relations 
    \begin{align*}
        \sum_{i+j=n}(-1)^id_id_j=0\ \text{ for all }n\ge 0
    \end{align*}
    instead of those in \eqref{diff relations}. 
    
    Finally, in the context of the de Rham complex on a Carnot group $G$, weights of differential forms are always non-negative. This means that the space of smooth forms $\Omega^\bullet(G)$ is actually an $(\N_0,\Z)$-graded $C^\infty(G)$-module. Not only this, but since the group $G$ is finite-dimensional, the range for the weights $a$ is finite and takes integer values between 0 and the weight $Q$ of the volume form. 
\end{remark}

\begin{definition}[Definition 2.2 in \cite{livernet2020spectral}]\label{total complex}
    For an $s$-multicomplex $\C$, its \textit{associated total complex} is defined as $\left(\mathrm{Tot}\,\C,d\right)$ with
    \begin{align}\label{ass tot}
        \left(\mathrm{Tot}\,\C\right)_h=\left(\prod_{\substack{a+b=h\\a\le0}}\C_{a,b}\right)\oplus\left(\bigoplus_{\substack{a+b=h\\ a>0}}\C_{a,b}\right)=\left(\bigoplus_{\substack{a+b=h\\b\le 0}}\C_{a,b}\right)\oplus\left(\prod_{\substack{a+b=h\\b>0}}\C_{a,b}\right)
    \end{align}
    and differential $d$ on $\mathrm{Tot}\,\C$ given for an arbitrary element $x\in\left(\mathrm{Tot}\,\C\right)$ by
    \begin{align}\label{diff on ass tot}
        \left(dx\right)_a=\sum_{i=0}^sd_i\left(x\right)_{a-i}\,.
    \end{align}

    Borrowing the notation used in the context of Carnot groups, we can also write the differential on $\mathrm{Tot}\,\C$ as $d=d_0+d_1+\cdots+d_s$, which turns out to be a complex by \eqref{diff relations} since
    \begin{align*}
        d^2=(d_0+d_1+\cdots+d_s)(d_0+d_1+\cdots+d_s)=\sum_{n=0}^{2s}\left(\sum_{i+j=n}d_{i}d_j\right)=0\,.
    \end{align*}
\end{definition}
In general, when working with $\left(\mathrm{Tot}\,\C\right)_h$, it is not always possible to consider the direct product total complex $\prod_{a+b=h}\C_{a,b}$ in degree $h$, as the formula \eqref{ass tot} may involve infinite sums. Notice, however, that in the case of the de Rham complex on Carnot groups this would not be an issue (see Remark \ref{remark 1.2}). Therefore, in this particular setting the associated total complex takes the simpler form
\begin{align*}
    \left(\mathrm{Tot}\,\C\right)_h=\bigoplus_{\substack{a+b=h\\0\le a\le Q}}\C_{a,b}\ \text{ with differential }d=d_0+d_1+\cdots+d_s\,.
\end{align*}

We are now going to follow the construction introduced by Rumin in \cite{rumin_grenoble,rumin_palermo} to extract the subcomplex $(E_0^\bullet,d_c)$, most commonly known as the Rumin complex, from $(\mathrm{Tot}\,\C,d=d_0+d_1+\cdots+d_s)$, the total complex associated with our $s$-multicomplex $\C$.

\subsection{Introducing a scalar product on $\C$}\label{scalar product}The space of Rumin forms $E_0^\bullet$ is isomorphic to the cohomology of the complex $(\mathrm{Tot}\,\C,d_0)$. However, since we want to define the $E_0^\bullet$ as subspaces of $\C$ instead of quotients, we require a way of identifying complements of the subspace $\mathrm{Im}\,d_0$. A natural solution is to introduce a scalar product on $\C$ and then take the orthogonal complement of $\mathrm{Im}\,d_0$.

In our setting, it is sufficient to define a scalar product $\langle\cdot,\cdot\rangle_h$ on each $(\mathrm{Tot}\,\C)_h$, such that elements of different weight are orthogonal, i.e. for each $h$ we have 
\begin{align*}
    \text{given } x_1\in\C_{a_1,b_1} \text{ and } x_2\in\C_{a_2,b_2} \text{ with }a_1+b_1=a_2+b_2=h\ ,\ \text{ if }a_1\neq a_2\text{ then }\langle x_1,x_2\rangle_h=0\,.
\end{align*}
Given a subspace $S$ of $\mathrm{Tot}\,\C$, we will denote the orthogonal projection onto $S$ by $\mathrm{pr}_S$.
\begin{definition}[The adjoint of $d_0$]\label{adjoint of d0} We can use the scalar product just introduced to define the formal transpose (adjoint) of $d_0$, which we will denote by $\delta_0$. In other words, $\delta_0\colon \C_{a,b}\to\C_{a,b-1}$ is defined by imposing
\begin{align*}
    \langle d_0 x,y\rangle_{{h+1}}=\langle x,\delta_0 y\rangle_{h}\ \text{ for any }x\in\left(\mathrm{Tot\,}\C\right)_h\text{ and }y\in\left(\mathrm{Tot\,}\C\right)_{h+1}\,. 
\end{align*}
The fact that $w(\delta_0x)=w(x)$ is a direct consequence of the fact that elements of different weight are orthogonal (see Lemma \ref{lemma 3.8} for the complete proof).    
\end{definition}

\begin{remark}\label{d0 closed range}
    Throughout this paper, we will be assuming that the map $d_0\colon\C_{a,b}\to\C_{a,b+1}$ has closed range. In the case of the de Rham complex on a Carnot group, the map $d_0$ is an extension of the Chevalley-Eilenberg differential to the space of smooth forms, so in particular it is a linear map between finite dimensional spaces (we refer to \cite{F+T1} for a thorough explanation of this fact). This assumption will be crucial when considering the properties of the Laplacian associated with $d_0$.
\end{remark}

\begin{definition}[Rumin forms on $(\mathrm{Tot}\,\C,d)$] Given the associated total complex $\mathrm{Tot}\,\C$ with differential $d=d_0+d_1+\cdots+d_s$, it easily follows from the definition of an $s$-multicomplex that $(\mathrm{Tot}\,\C,d_0)$ is itself a complex, i.e. $d_0^2=0$. Therefore, it is possible to consider its cohomology. Since we are interested in subspaces and not quotients, we are going to introduce the space of Rumin forms $E_0^\bullet$ associated with the total complex $(\mathrm{Tot}\,\C,d=d_0+d_1+\cdots+d_s)$ as
\begin{align}
    E_0^h=\ker d_0 \,\cap \left(\mathrm{Im}\,d_0\right)^\perp\cap \left(\mathrm{Tot}\,\C\right)_h=\ker d_0\cap \ker \delta_0\cap\left(\mathrm{Tot}\,\C\right)_h\,,
\end{align}
the last equality following from the closed range theorem (see Remark \ref{d0 closed range}).

\end{definition}

Another crucial operator that can be defined from $d_0$ using the scalar product is its partial inverse. This operator will be central in the construction of the so-called Rumin differential $d_c$.
\begin{definition}[The partial inverse $d_0^{-1}$]\label{d_0 inverse} The map $d_0$ being linear implies that it acts as a bijection from $\left(\ker d_0\right)^\perp=\mathrm{Im}\,\delta_0$ onto $\mathrm{Im}\,d_0$. Following Rumin's construction, it is customary to use the shorthand notation $d_0^{-1}$ to denote the linear map given by 
\begin{align*}
    d_0^{-1}:=d_0^{-1}\mathrm{pr}_{\mathrm{Im}\,d_0}\colon \C_{a,b}\to \C_{a,b-1}\,.
\end{align*}
In particular, it readily follows that $(d_0^{-1})^2=0$ and $\ker d_0^{-1}=\ker\delta_0=\left(\mathrm{Im\,}d_0\right)^\perp$.
\end{definition}
\begin{definition}[Orthogonal projection onto $E_0^\bullet$] Let us consider the operator
\begin{align*}
    \Pi_0:=\mathrm{Id}-d_0^{-1}d_0-d_0d_0^{-1}\colon \C_{a,b}\longrightarrow\C_{a,b}\,.
\end{align*}
From the definition of the map $d_0^{-1}$, we have 
\begin{itemize}
    \item $\Pi_0$ maps elements of $\C_{a,b}$ onto elements of $\C_{a,b}$, i.e. it preserves both the weight and the degree;
    \item $d_0^{-1}d_0=\mathrm{pr}_{\mathrm{Im}\,\delta_0}$ and $d_0d_0^{-1}=\mathrm{pr}_{\mathrm{Im}\,d_0}$;
    \item $\Pi_0^2=\left(\mathrm{Id}-d_0^{-1}d_0-d_0d_0^{-1}\right)\left(\mathrm{Id}-d_0^{-1}d_0-d_0d_0^{-1}\right)=\mathrm{Id}-d_0^{-1}d_0-d_0d_0^{-1}=\Pi_0$.
\end{itemize}
   Therefore, the map $\Pi_0$ is an orthogonal projection onto 
   \begin{align*}
       \mathrm{Im\,}\Pi_0=\ker d_0\cap\ker d_0^{-1}=\ker d_0\cap \ker \delta_0=E_0\,,
   \end{align*}
   that is $\Pi_0=\mathrm{pr}_{E_0}$.
\end{definition}

We are interested in characterising the space of Rumin forms also as ``harmonic forms'' of the Laplacian associated with $d_0$ and its adjoint $\delta_0$. We will also see how, under the assumption of $d_0$ having a closed range, such a Laplacian admits a Hodge-decomposition.

\begin{definition}[The Laplacian $\Box_0$] Let us consider the Laplacian operator defined using $d_0$ and its adjoint
\begin{align*}
    \Box_0:=d_0\delta_0+\delta_0d_0\colon\C_{a,b}\longrightarrow\C_{a,b}\,.
\end{align*}
Just like the projection $\Pi_0$, $\Box_0$ preserves both the weight and the degree. It is a symmetric operator, that is $\langle\Box_0x,y\rangle_h=\langle x,\Box_0y\rangle_h$ for any $x,y\in\left(\mathrm{Tot}\,\C\right)_h$, and its kernel coincides with the space of Rumin forms
\begin{align*}
    \ker\Box_0=\ker d_0\cap\ker\delta_0=E_0\,.
\end{align*}
The inclusion $\ker d_0\cap\ker\delta_0\subseteq\ker\Box_0$ is clear,  while the converse one follows from the fact that
\begin{align*}
    0=\langle\Box_0x,x\rangle_h=\langle d_0\delta_0x,x\rangle_h+\langle\delta_0d_0x,x\rangle_h=\vert d_0x\vert_h^2+\vert \delta_0x\vert^2_h\ \ \text{for any }x\in\left(\mathrm{Tot}\,\C\right)_h\,.
\end{align*}
\end{definition}

\begin{proposition}[Hodge decomposition for $\Box_0$]\label{lemma hodge decomp}
    The associated total complex $\mathrm{Tot}\,\C$ admits a direct sum (or a Hodge) decomposition in terms of $\Box_0$, $d_0$, and its adjoint $\delta_0$. More explicitly,
    \begin{align}\label{eq hodge decomp}
        \mathrm{Tot}\,\C=\mathrm{Im}\,d_0\oplus\ker\Box_0\oplus\mathrm{Im}\,\delta_0
    \end{align}
\end{proposition}
\begin{proof}
    The statement follows from the fact that the orthogonal complement of the kernel of a map with closed range coincides with the range of its adjoint:
    \begin{align*}
        \mathrm{Tot}\,\C=&\ker d_0\oplus\left(\ker d_0\right)^\perp=\ker d_0\oplus\mathrm{Im}\,\delta_0\\
        =&\ker d_0\cap \ker\delta_0\oplus\ker d_0\cap\left(\ker\delta_0\right)^\perp\oplus\mathrm{Im}\,\delta_0\\=&\ker\Box_0\oplus\ker d_0\cap\mathrm{Im}\,d_0\oplus\mathrm{Im}\,\delta_0=\ker\Box_0\oplus\mathrm{Im}\,d_0\oplus\mathrm{Im}\,\delta_0
    \end{align*}
\end{proof}
\begin{remark}One should notice that in the specific case of an $s$-multicomplex there is a clear symmetry in the operators (see \cite{cirici2025modelcategorystructurestruncated} for a more thorough exploration of this symmetry). For a truncated multicomplex, we clearly have two operators $d_0$ and $d_s$ that satisfy the condition $d_0^2=d_s^2=0$, so a priori one could repeat exactly the same construction to obtain a Hodge decomposition using the $d_s$ operator. However, in the application case that we have in mind, namely differential geometry on Carnot groups, this symmetry is broken by the fact that $d_0$ and $d_s$ have different properties. In this subRiemannian setting, $d_0$  is a linear operator that acts between finite dimensional spaces, hence Remark \ref{d0 closed range} is always true. On the other hand, $d_s$ is a differential operator acting on the space of smooth forms, so in general its range will not be closed. This is the main reason why we choose to present here a Hodge decomposition of the space of $\mathrm{Tot}\,\C$ in terms of $d_0$.
\end{remark}

\subsection{Constructing the Rumin complex}
We have already defined the space of Rumin forms $E_0^\bullet$ as the space of harmonic forms for the $\Box_0$ Laplacian, which is isomorphic to the cohomology of the complex $(\mathrm{Tot}\,\C,d_0)$. We are now ready to define a cochain map $d_c\colon E_0^h\to E_0^{h+1}$ such that the cohomology of the complex $(E_0^\bullet,d_c)$ is isomorphic to the cohomology of $(\mathrm{Tot}\,\C,d=d_0+d_1+\cdots+d_s)$.

Notice that in order for this construction to work, we are assuming that the range for the weights $a$ is finite (see Remark \ref{remark 1.2}).

\begin{lemma}\label{increase weight implies nilpotent}
    If $B\colon\mathrm{Tot}\,\C\to\mathrm{Tot}\,\C$ is an operator that strictly increases the weight, i.e.
    \begin{align*}
        \text{given }x\in\C_{a,\bullet}\,,\ \text{we have that }Bx\in\C_{a',\bullet}\text{ with }a'>a\,,
    \end{align*}
    then $B$ is nilpotent, i.e. there exists $N\in\mathbb{N}$ independent of $B$ such that $B^N=0$.
\end{lemma}
\begin{proof}
    The nilpotency of any operator that strictly increases the weight of elements in $\mathrm{Tot}\,\C$ simply follows from the fact that we are assuming the range of the weights is finite.
\end{proof}
\begin{lemma}\label{inverse of nilpotent}
    Let $A_0\colon\mathrm{Tot}\,\C\to\mathrm{Tot}\,\C$ be an invertible linear map that preserves the weights, while $B\colon\mathrm{Tot}\,\C\to\mathrm{Tot}\,\C$ is an operator that strictly increases the weights. Then the operator $A:=A_0-B\colon\mathrm{Tot}\,\C\to\mathrm{Tot}\,\C$ is invertible and its inverse $A^{-1}\colon\mathrm{Tot}\,\C\to\mathrm{Tot}\,\C$ has the form
    \begin{align*}
        A^{-1}=A_0^{-1}\sum_{j=0}^{N-1}\left( BA_0^{-1}\right)^j=\left(\sum_{j=0}^{N-1}\left(A_0^{-1}B^j\right)\right)A_0^{-1}\,,
    \end{align*}
    with $N\in\mathbb{N}$ being the nilpotency exponent of $B$ as discussed in Lemma \ref{increase weight implies nilpotent}.
\end{lemma}
\begin{proof}
    By construction, we have $A=A_0-B=(\mathrm{Id}-BA_0^{-1})A_0=A_0(\mathrm{Id}-A_0^{-1}B)$. Both the operators $BA_0^{-1}$ and $A_0^{-1}B$ act on $\mathrm{Tot\,}\C$ by strictly increasing weights, hence they are both nilpotent by Lemma \ref{increase weight implies nilpotent}. As a consequence, we get invertibility via Neumann series: both $\mathrm{Id}-BA_0^{-1}$ and $\mathrm{Id}-A_0^{-1}B$ are invertible and
    \begin{align*}
        (\mathrm{Id}-BA_0^{-1})^{-1}=\sum_{j=0}^{N-1}\left(BA_0^{-1}\right)^j\ \text{ and }\ (\mathrm{Id}-A_0^{-1}B)^{-1}=\sum_{j=0}^{N-1}\left(A_0^{-1}B\right)^j\,.
    \end{align*}
    The claim follows directly from the fact that $A^{-1}=A_0^{-1}(\mathrm{Id}-BA_0^{-1})^{-1}=(\mathrm{Id}-A_0^{-1}B)^{-1}A_0^{-1}$.
\end{proof}
Using the previous results, we are now able to construct the necessary projection maps used to define the cochain map $d_c$.
\begin{lemma}\label{lemma 1.14}
    Let us consider the operator $b:=d_0^{-1}d_0-d_0^{-1}d=-d_0^{-1}(d-d_0)$ acting on $\mathrm{Tot\,}\C$.
    \begin{itemize}
        \item[1.] The map $b\colon\mathrm{Tot}\,\C\longrightarrow\mathrm{Tot}\,\C$ preserves the degree, it is nilpotent and so $\mathrm{Id}-b$ is invertible.
        \item[2.] The operator 
        \begin{align}\label{piF}
            \Pi:=(\mathrm{Id}-b)^{-1}d_0^{-1}d+d(\mathrm{Id}-b)^{-1}d_0^{-1}
        \end{align}
        is the projection of $\mathrm{Tot}\,\C$ onto
        \begin{align*}
            F:=\mathrm{Im}\,d_0^{-1}+\mathrm{Im}\,dd_0^{-1}=\mathrm{Im}\,\delta_0+\mathrm{Im}\,d\delta_0
        \end{align*}
        along
        \begin{align*}
            E:=\ker d_0^{-1}\cap\ker d_0^{-1}d=\ker\delta_0\cap\ker\delta_0d\,.
        \end{align*}
    \end{itemize}
\end{lemma}
\begin{proof}$\phantom{=}$
    \begin{itemize}
        \item [1.] By definition, the operator $b$ strictly increases the weight and so by Lemma \ref{increase weight implies nilpotent} it is nilpotent, i.e. there exists $N\in\mathbb{N}$ such that $b^N=0$. We can now apply Lemma \ref{inverse of nilpotent} with $A_0=\mathrm{Id}$ and $B=b$ so the operator $\mathrm{Id}-b$ is invertible and its inverse if given by
        \begin{align}\label{Id-b inverse}
            \left(\mathrm{Id}-b\right)^{-1}=\sum_{j=0}^{N-1}b^j=\sum_{j=0}^{N-1}\left[-d_0^{-1}(d-d_0)\right]^{j}\,.
        \end{align}
        This also implies that $\Pi$ as defined in \eqref{piF} is a well-defined operator on $\mathrm{Tot}\,\C$.
        \item[2.] By \eqref{piF} and \eqref{Id-b inverse}, we have that the image of $\Pi$ is included in the subspace $F\subset\mathrm{Tot}\,\C$ defined in the statement. Moreover, using the properties of $d$, $d_0$, and $d_0^{-1}$ we get that
        \begin{align*}
            \Pi d_0^{-1}=d_0^{-1}\ \text{ and }\ \Pi dd_0^{-1}=dd_0^{-1}
        \end{align*}
        so that $\Pi=\mathrm{Id}$ on $F$ and $\Pi^2=\Pi$, i.e. $\Pi$ is a projection onto $\mathrm{Im}\,\Pi=F$ along $\ker\Pi$. Finally, $\ker \Pi$ contains the subspace $E$ introduced in the statement. Since we also have the equalities
        \begin{align*}
            d_0^{-1}\Pi=d_0^{-1}\ \text{ and }\ d_0^{-1}d\Pi=d_0^{-1}d\,,
        \end{align*}
        we obtain that $\ker\Pi=E$.
    \end{itemize}
\end{proof}
Following Rumin's notation, we denote the two projections onto $F$ and $E$ respectively as
\begin{align*}
    \Pi_F:=\Pi\ \text{ and }\ \Pi_E:=\mathrm{Id}-\Pi\,,
\end{align*}
where $\Pi$ is the operator defined in \eqref{piF}.

\begin{lemma}\label{properties of proj} The projector operators $\Pi_E$, $\Pi_F$, and $\Pi_0$ enjoy the following properties:
\begin{itemize}
    \item [1.]$d_0^{-1}\Pi_E=\Pi_Ed_0^{-1}=0$;
    \item[2.]$d\Pi_F=\Pi_Fd$ and $d\Pi_E=\Pi_Ed$;
    \item[3.] $\Pi_0\Pi_E\Pi_0=\Pi_0$ and $\Pi_E\Pi_0\Pi_E=\Pi_E$.
\end{itemize}
    
\end{lemma}
\begin{proof}The equalities in 1. follow from the fact that $d_0^{-1}\Pi=\Pi d_0^{-1}$, while those in 2. can easily be checked by carrying out the explicit computations that show $d\Pi=d(\mathrm{Id}-b)^{-1}d_0^{-1}d=\Pi d$.

Finally, to get the equalities in 3. we need to apply the fact that $\mathrm{Im\,}\Pi\subset \mathrm{Im}\,d_0^{-1}\subset (\ker\Box_0)^\perp$ and the equalities in 1., so that
\begin{align*}
    \Pi_E\Pi_{0}\Pi_E=&\Pi_E^2-\Pi_E(d_0d_0^{-1}+d_0^{-1}d_0)\Pi_E=\Pi_E\ \text{ and }\ \Pi_0\Pi_E\Pi_0=\Pi_0^2-\Pi_0\Pi\,\Pi_0=\Pi_0\,.
\end{align*}
    
\end{proof}
Lemma \ref{properties of proj} implies the following result.
\begin{proposition}[Theorem 2.6 in \cite{rumin_grenoble}]
    The associated total complex $(\mathrm{Tot}\,\C,d=d_0+d_1+\cdots+d_s)$ splits into two subcomplexes $(E^\bullet,d)$ and $(F^\bullet,d)$. Moreover, the following operator
    \begin{align*}
        d_c:=\Pi_0d\Pi_E\Pi_0\colon E_0^\bullet\subset\mathrm{Tot}\,\C\longrightarrow E_0^\bullet\subset\mathrm{Tot}\,\C
    \end{align*}
    satisfies 
    \begin{align*}
        d_c^2=0\ \text{ and for any }x\in(\mathrm{Tot}\,\C)_h\ d_cx\in(\mathrm{Tot}\,\C)_{h+1}\,.
    \end{align*}
    Finally, the complex $(E_0^\bullet,d_c)$ computes the same cohomology as the total complex $(\mathrm{Tot}\,\C,d)$ and it is known as the Rumin complex
\end{proposition}

\begin{proof}
    The claim follows easily from the fact that, by property 2. in Lemma \ref{properties of proj}, $\Pi_E$ is a homotopical equivalence between the total complex $(\mathrm{Tot}\,\C,d)$ and the complex $(E^\bullet,d)$. Finally, property 3. in Lemma \ref{properties of proj} further implies that $E$ and $E_0$ are in bijection, and that $\Pi_E$ restricted to $E_0$ and $\Pi_0$ restricted to $E$ are inverse maps of each other. Hence the complex $(E^\bullet,d)$ is conjugated to $(E_0^\bullet,d_c)$ with $d_c=\Pi_0d\Pi_E\Pi_0$, as defined in the statement.
\end{proof}
\subsection{Introducing a shorthand notation for the operator $d_c$}
Before proceeding to express the spaces arising from the spectral sequence associated to our $s$-multicomplex, we are going to study the explicit expression of the $d_c$. In the process, we will introduce a new shorthand notation in order to make the computations needed in the next sections more accessible.
\begin{definition}\label{def partial r}Given the structure maps $d_0,d_1,\ldots,d_s$ of our $s$-multicomplex together with the partial inverse $d_0^{-1}$ of Definition \ref{d_0 inverse}, we are going to define by induction the following maps on $\mathrm{Tot}\,\C$:
\begin{align}
    \partial_1=d_1\ \text{ and }\partial_r=d_r-\sum_{j=1}^{r-1}d_{r-j}d_0^{-1}\partial_j\ \text{ for }r\ge 2.
\end{align}
\end{definition}
For clarity, let us see the explicit expression of the operator $\partial_3$:
\begin{align*}
    \partial_3=&d_3-\sum_{j=1}^2d_{3-j}d_0^{-1}\partial_j=d_3-d_2d_0^{-1}\partial_1-d_1d_0^{-1}\partial_2=d_3-d_2d_0^{-1}d_1-d_1d_0^{-1}\left(d_2-d_1d_0^{-1}\partial_1\right)\\=&d_3-d_2d_0^{-1}d_1-d_1d_0^{-1}d_2+d_1d_0^{-1}d_1d_0^{-1}d_1
\end{align*}
\begin{remark}
    Notice that by definition each operator $\partial_r$ has bidegree $\vert \partial_r\vert=(r,1-r)$, that is it increases the weight by $r$ and the degree by 1:
    \begin{align*}
        \text{ for any }x\in \C_{a,b}\ \text{we have }\partial_rx\in \C_{a+r,b+1-r}\,.
    \end{align*}
\end{remark}
\begin{lemma}[Expressing $d_c$ on $E_0$]\label{expression of d_c} 
    Using the operators introduced in Definition \ref{def partial r}, one can express the operator $d_c$ acting on $E_0$ as follows:
    \begin{align}\label{the rumin diff}
        d_cx=\Pi_0\sum_{r=1}^{N-1}\partial_rx=\sum_{r=1}^{N-1}\partial_rx-d_0^{-1}d_0\sum_{r=1}^{N-1}\partial_rx-d_0d_0^{-1}\sum_{r=1}^{N-1}\partial_rx\ \text{ for any }x\in E_0
    \end{align}
    where $N\in\mathbb{N}$ is independent of $x$ or $d$ and it is such that $\partial_Nx=0$.
\end{lemma}
\begin{proof}
    Let us first focus on the action of $\Pi$ on elements of $E_0$. By definition, if $x\in E_0$ then $x\in\ker d_0\cap\ker \delta_0=\ker d_0\cap \ker d_0^{-1}$ and so the action of $\Pi$ simplifies to \begin{align*}
        \Pi x=&(\mathrm{Id}-b)^{-1}d_0^{-1}dx=\sum_{j=0}^{N-1}\left[-d_0^{-1}(d-d_0)\right]^jd_0^{-1}(d-d_0)x=\sum_{j=1}^{N-1}(-1)^{j-1}[d_0^{-1}(d-d_0)]^jx\\=&d_0^{-1}(d-d_0)x-[d_0^{-1}(d-d_0)]^2x+[d_0^{-1}(d-d_0)]^3x+\cdots+(-1)^{N-2}[d_0^{-1}(d-d_0)]^{N-1}x\\=&\underbrace{d_0^{-1}d_1}_{d_0^{-1}\partial_1}x+\underbrace{d_0^{-1}\left(d_2-d_1d_0^{-1}d_1\right)}_{d_0^{-1}\partial_2}x+\underbrace{d_0^{-1}\left(d_3-d_2d_0^{-1}d_1-d_1d_0^{-1}d_2+d_1d_0^{-1}d_1d_0^{-1}d_1\right)}_{d_0^{-1}\partial_3}x+\cdots\\=&\sum_{r=1}^{N-1}d_0^{-1}\partial_rx
    \end{align*}
    so that 
    \begin{align*}
        d\Pi_Ex=&d\left(x-\sum_{r=1}^{N-1}d_0^{-1}\partial_rx\right)=(d-d_0)x-d_0d_0^{-1}\sum_{r=1}^{N-1}\partial_rx-(d-d_0)\sum_{r=1}^{N-1}d_0^{-1}\partial_rx\\=&\sum_{j=1}^sd_jx-\sum_{j=1}^sd_jd_0^{-1}\sum_{r=1}^{N-1}\partial_rx-d_0d_0^{-1}\sum_{r=1}^{N-1}\partial_rx\\
        =&d_1x+d_2x+\cdots+d_sx-d_1d_0^{-1}\sum_{r=1}^{N-1}\partial_rx-d_2d_0^{-1}\sum_{r=1}^{N-1}\partial_rx-\cdots-d_sd_0^{-1}\sum_{r=1}^{N-1}\partial_rx-d_0d_0^{-1}\sum_{r=1}^{N-1}\partial_rx\\=&\underbrace{d_1x}_{\partial_1x}+\underbrace{d_2x-d_1d_0^{-1}\partial_1x}_{\partial_2x}+\underbrace{d_3x-d_1d_0^{-1}\partial_2-d_2d_0^{-1}\partial_1x}_{\partial_3x}+\cdots-d_0d_0^{-1}\sum_{r=1}^{N-1}\partial_rx\\=&\sum_{r=1}^{N-1}\partial_rx-d_0d_0^{-1}\sum_{r=1}^{N-1}\partial_rx=\sum_{r=1}^{N-1}\partial_rx-\mathrm{pr}_{\mathrm{Im}\,d_0}\sum_{r=1}^{N-1}\partial_rx=\mathrm{pr}_{(\mathrm{Im}\,d_0)^\perp}\sum_{r=1}^{N-1}\partial_rx\,.
    \end{align*}
    Finally, since $d_0^{-1}$ acts trivially on $(\mathrm{Im}\,d_0)^\perp$ and we have just shown that $d\Pi_Ex\in(\mathrm{Im}\,d_0)^\perp$, we get
    \begin{align*}
        \Pi_0d\Pi_Ex=&\left(\mathrm{Id}-d_0^{-1}d_0\right)d\Pi_Ex=\left(\mathrm{Id}-d_0^{-1}d_0\right)\left(\sum_{r=1}^{N-1}\partial_rx-d_0d_0^{-1}\sum_{r=1}^{N-1}\partial_rx\right)\\=&\sum_{r=1}^{N-1}\partial_rx-d_0^{-1}d_0\sum_{r=1}^{N-1}\partial_rx-d_0d_0^{-1}\sum_{r=1}^{N-1}\partial_rx\,.
    \end{align*}
\end{proof}
\begin{remark}
    The fact that $d\Pi_Ex$ belongs to $(\mathrm{Im}\,d_0)^\perp$ also follows from the fact that $d\Pi_Ex=\Pi_Edx\in E\subset\ker\delta_0=(\mathrm{Im}\,d_0)^\perp$ by Lemma \ref{lemma 1.14}.

    Moreover, in general, there will be no relationship between $s$ and $N$. Depending on the multicomplex considered, as well as both the weight and degree of the element $x\in\mathrm{Tot}\,\C$ considered, one could have $s$ bigger, smaller or equal to $N$. The only thing that is certain is that they are both finite positive integers. For example, in the case where $N-1>s$ we will have that 
    \begin{align*}
        \partial_rx=d_rx-\sum_{j=1}^{r-1}d_{r-j}d_0^{-1}\partial_jx=-\sum_{j=r-s}^{r-1}d_{r-j}d_0^{-1}\partial_{j}\ \text{ for any }r>s\,.
    \end{align*}
\end{remark}
Since $d_0d_0^{-1}$ and $d_0^{-1}d_0$ are both projections of bidegree $(0,0)$, i.e. they keep both the weight and the degree constant, the operator $d_c$ consists of a sum of operators $\partial_r-d_0d_0^{-1}\partial_r-d_0^{-1}d_0\partial_r$ of bidegree $\vert\partial_r\vert=(r,1-r)$. Since we will be interested in each one of such operators, we introduce the following notation.
\begin{definition}\label{d_c^r}
    Given the Rumin differential $d_c$ with expression as in \eqref{the rumin diff}, we introduce the notation
    \begin{align*}
        d_c^r=\partial_r-d_0^{-1}d_0\partial_r-d_0d_0^{-1}\partial_r\ \text{ for each }r=1,\ldots,N-1
    \end{align*}
    to denote each addend of $d_c$ of bidegree $\vert d_c^r\vert=(r,1-r)$.
\end{definition}
\section{Characterising spectral sequences in terms of the Rumin differentials}\label{section 3}

In this section, we focus on the spectral sequence arising from an $s$-multicomplex. Our main goal is to clarify the structure of this spectral sequence by showing that the graded modules $Z_r^{p,\bullet}$ and $B_r^{p,\bullet}$, introduced in \cite{livernet2020spectral} to describe the quotient spaces $E_r^{p,\bullet}$ at each page, admit an equivalent formulation in terms of Rumin forms and the action of the Rumin differential $d_c$. For the sake of completeness and to make the discussion self-contained, we recall below the definitions of the modules $Z_r^{p,\bullet}$ and $B_r^{p,\bullet}$, with the indexing adapted to agree with the bidegree convention adopted in Definition~\ref{def multicomplex}.

\begin{definition}[Definition 2.6 in \cite{livernet2020spectral}]\label{Z and B defined} Let $\alpha\in\C_{p,\bullet}$ and let $r\ge 1$. We define subgraded modules $Z_r^{p,\bullet}$ and $B_r^{p,\bullet}$ of $\C_{p,\bullet}$ as follows.
\begin{align*}
    \alpha\in Z_r^{p,\bullet}\ \Longleftrightarrow&\ \text{for }1\le j\le r-1\,,\text{ there exists }z_{p+j}\in\C_{p+j,\bullet}\text{ such that}\\&\ d_0\alpha=0\text{ and }d_n\alpha=\sum_{i=0}^{n-1}d_iz_{p+n-i}\text{ for all }1\le n\le r-1\,.\\
    \alpha\in B_r^{p,\bullet}\ \Longleftrightarrow&\ \text{for }0\le k\le r-1\text{ there exists }c_{p-k}\in\C_{p-k,\bullet}\text{ such that}\\&\ \alpha=\sum_{k=0}^{r-1}d_kc_{p-k}\text{ and }0=\sum_{k=l}^{r-1}d_{k-l}c_{p-k}\text{ for }1\le l\le r-1\,.
\end{align*}
\end{definition}

As recalled in Proposition \ref{lemma hodge decomp}, we have a Hodge-decomposition coming from the Laplacian $\Box_0=d_0\delta_0+\delta_0d_0$ for any element in $\mathrm{Tot}\,\C$. To make this extra structure explicit, we introduce the following notation.
\begin{definition}
    Let $\alpha\in\mathrm{Tot}\,\C$. Following the Hodge-decomposition described in \eqref{eq hodge decomp}, we have
    \begin{align*}
        \alpha=\check{\alpha}+\overline{\alpha}+\hat{\alpha}
    \end{align*}
    where
    \begin{itemize}
        \item $\check{\alpha}=d_0d_0^{-1}\alpha\in\mathrm{Im}\,d_0$;
        \item $\hat{\alpha}=d_0^{-1}d_0\alpha\in\mathrm{Im}\,\delta_0$;
        \item $\overline{\alpha}=\Pi_0\alpha=\alpha-d_0d_0^{-1}\alpha-d_0^{-1}d_0\alpha\in\ker\Box_0$.
    \end{itemize}
    
\end{definition}
 Throughout this section, for any $\alpha\in\C_{p,\bullet}$, we will use the notation $\beta_p$ in the case where $\check{\alpha}=d_0\beta_p$, as a way to keep track of the weight of the element $\beta_p$.

    Notice that in general such an element $\beta_p\in\C_{p,\bullet}$ will not belong just to $\mathrm{Im}\,\delta_0$. This follows from the fact that solutions of $d_0\beta_p=\check{\alpha}$ are determined up to adding an element of $\ker d_0$. Since $\ker d_0=\ker\Box_0\oplus\mathrm{Im\,}d_0$, we have that
    \begin{align*}
        \beta_p=\check{\beta}_p+\overline{\beta}_p+\hat{\beta}_p\ \text{ with }\ \hat{\beta}_p=d_0^{-1}d_0\beta_p=d_0^{-1}\alpha=d_0^{-1}\check{\alpha}\,.
    \end{align*}
    Notice that the harmonic representatives for such an algebraic Laplacian are indeed Rumin forms. For example, for any $\alpha\in\C_{p,\bullet}$, we have that $\overline{\alpha}=\Pi_0\alpha\in E_0$ is a Rumin form (unless trivial).

    Finally, the following formula will come in handy for the computations of the rest of the section.

    \begin{lemma}\label{d0 of partialr}
        Given an arbitrary element $\alpha\in\C_{p,\bullet}$, then $\partial_1\overline{\alpha}\in\ker d_0$, i.e. $d_0\partial_1\overline{\alpha}=0$, and for any $r\ge 2$
        \begin{align}\label{formula d0 of partial r}
            d_0\partial_r\overline{\alpha}=-\sum_{i=1}^{r-1}d_i\left(\partial_{r-i}-d_0d_0^{-1}\partial_{r-i}\right)\overline{\alpha}\,.
        \end{align}
    \end{lemma}
    \begin{proof}
        The first claim readily follows from the fact that $\overline{\alpha}\in\ker\Box_0\cap\C_{p,\bullet}$, so that
        \begin{align*}
            d_0\partial_1\overline{\alpha}=&d_0d_1\overline{\alpha}=-d_1d_0\overline{\alpha}=0\,.
        \end{align*}
        Here we are using the relations of the structure maps \eqref{diff relations} together with the fact that $\overline{\alpha}\in\ker \Box_0\subset\ker d_0$.

        If $r=2$, then
        \begin{align*}
            d_0\partial_2\overline{\alpha}=&d_0\left(d_2-d_1d_0^{-1}d_1\right)\overline{\alpha}=-d_1^2\overline{\alpha}-d_2\underbrace{d_0\overline{\alpha}}_{=0}+d_1d_0d_0^{-1}d_1\overline{\alpha}=-d_1\left(d_1\overline{\alpha}-d_0d_0^{-1}d_1\overline{\alpha}\right)\,.
        \end{align*}
        In general, by the relations of the structure maps \eqref{diff relations}, we have
        \begin{align}\label{d0dr}
            d_0d_r=-\sum_{i=1}^{r-1}d_id_{r-i}-d_rd_0\ \text{ for any }r\ge 1
        \end{align}
        and so
        \begin{align*}
            d_0\partial_r\overline{\alpha}=&d_0\left(d_r-\sum_{i=1}^{r-1}d_{r-i}d_0^{-1}\partial_i\right)\overline{\alpha}=d_0d_r\overline{\alpha}-\sum_{i=1}^{r-1}d_0d_{r-i}d_0^{-1}\partial_i\overline{\alpha}\\=&-\sum_{i=1}^{r-1}d_id_{r-i}\overline{\alpha}-d_rd_0\overline{\alpha}+\sum_{i=1}^{r-1}\sum_{j=1}^{r-i-1}d_jd_{r-i-j}d_0^{-1}\partial_i\overline{\alpha}+\sum_{i=1}^{r-1}d_{r-i}d_0d_0^{-1}\partial_i\overline{\alpha}\\=&-\sum_{i=1}^{r-1}d_id_{r-i}\overline{\alpha}+\sum_{j=1}^{r-2}\sum_{i=1}^{r-j-1}d_jd_{r-i-j}d_0^{-1}\partial_i\overline{\alpha}+\sum_{i=1}^{r-1}d_id_0d_0^{-1}\partial_{r-i}\overline{\alpha}\\=&-d_{r-1}d_1\overline{\alpha}+d_{r-1}d_0d_0^{-1}\partial_1\overline{\alpha}-\sum_{i=1}^{r-2}\left(d_id_{r-i}-\sum_{j=1}^{r-i-1}d_id_{(r-i)-j}d_0^{-1}\partial_j\right)\overline{\alpha}+\sum_{i=1}^{r-2}d_id_0d_0^{-1}\partial_{r-i}\overline{\alpha}\\=&-d_{r-1}\left(d_1-d_0d_0^{-1}d_1\right)\overline{\alpha}-\sum_{i=1}^{r-2}d_i\left(d_{r-i}-\sum_{j=1}^{r-i-1}d_{(r-i)-j}d_0^{-1}\partial_j-d_0d_0^{-1}\partial_{r-i}\right)\overline{\alpha}\\=&-\sum_{i=1}^{r-1}d_i\left(\partial_{r-i}-d_0d_0^{-1}\partial_{r-i}\right)\overline{\alpha}\,.
        \end{align*}
    \end{proof}

    \subsection{Expressing the spaces $Z_r^{p,\bullet}$ in terms of Rumin forms and the Rumin differential}
    \subsubsection*{r=1} Using Definition \ref{Z and B defined} for $r=1$, we have
    \begin{align*}
        \alpha\in Z_1^{p,\bullet}\ \Longleftrightarrow\ d_0\alpha=0\ \Longleftrightarrow\ \alpha=\check{\alpha}+\overline{\alpha}\text{ and }\hat{\alpha}=d_0^{-1}d_0\alpha=0\,.
    \end{align*}
    This can be rephrased more explicitly as
    \begin{align*}
    \alpha=d_0\beta_p+\overline{\alpha}\text{ for some }\beta_p\in\C_{p,\bullet}\,.
    \end{align*}
    \subsubsection*{r=2} Using Definition \ref{Z and B defined} for $r=2$, we have
    \begin{align*}
        \alpha\in Z_2^{p,\bullet}\ \Longleftrightarrow\ d_0\alpha=0\text{ and there exists a }z_{p+1}\in\C_{p+1,\bullet}\text{ such that }d_1\alpha=d_0z_{p+1}\,.
    \end{align*}
    Our claim is that the condition $\alpha\in Z_2^{p,\bullet}$ is equivalent to requiring $\alpha=d_0\beta_p+\overline{\alpha}$ and $d_c^1\overline{\alpha}=0$. Indeed,
    \begin{align*}
        d_1\alpha=d_1(d_0\beta_p+\overline{\alpha})=d_0z_{p+1}\ \Longleftrightarrow\ d_1\overline{\alpha}=d_0z_{p+1}-d_1d_0\beta_p=d_0(z_{p+1}+d_1\beta_p)\,,
    \end{align*}
where we are using the relations of the structure maps \eqref{diff relations}, namely $d_0d_1+d_1d_0=0$.

Moreover, by Lemma \ref{d0 of partialr}, we know that $d_1\overline{\alpha}\in\ker d_0$, but also 
\begin{align*}
    d_1\overline{\alpha}=d_0\omega_{p+1}\text{ with }\omega_{p+1}:=z_{p+1}+d_1\beta_p\ \Longrightarrow\ d_1\overline{\alpha}\in\mathrm{Im}\,d_0
\end{align*}
and so
\begin{align*}
    d_c^1\overline{\alpha}=d_1\overline{\alpha}-d_0d_0^{-1}d_1\overline{\alpha}-d_0^{-1}d_0d_1\overline{\alpha}=d_1\overline{\alpha}-d_0d_0^{-1}d_0\omega_{p+1}=d_1\overline{\alpha}-d_0\omega_{p+1}=0\,.
\end{align*}

To summarise, we know that $\alpha=d_0\beta_p+\overline{\alpha}$ for some $\beta_p\in\C_{p,\bullet}$ (which coincides with the condition $\alpha\in Z_1^{p,\bullet}$). Moreover, knowing that there exists $\omega_{p+1}=z_{p+1}+d_1\beta_p\in\C_{p+1,\bullet}$ such that $d_1\overline{\alpha}=d_0\omega_{p+1}$ and that $\partial_1\overline{\alpha}=d_1\overline{\alpha}\in\ker d_0$, we get that
\begin{align*}
    d_1\overline{\alpha}=d_1\overline{\alpha}-&d_0d_0^{-1}d_1\overline{\alpha}+d_0d_0^{-1}d_1\overline{\alpha}=d_c^1\overline{\alpha}+d_0d_0^{-1}d_1\overline{\alpha}=d_0\omega_{p+1}\\
    &\Longleftrightarrow\ \underbrace{d_c^1\overline{\alpha}}_{\in\ker\Box_0}=\underbrace{d_0(\omega_{p+1}-d_0^{-1}d_1\overline{\alpha})}_{\in\mathrm{Im}\,d_0}\,.
\end{align*}
By the direct sum decomposition in \eqref{eq hodge decomp} we have that both the right and left hand sides must vanish, i.e.
\begin{itemize}
    \item $d_c^1\overline{\alpha}=0$;
    \item $d_0(\omega_{p+1}-d_0^{-1}d_1\overline{\alpha})=0$, which means that $\hat{\omega}_{p+1}-d_0^{-1}d_1\overline{\alpha}=d_0^{-1}d_0\omega_{p+1}-d_0^{-1}d_1\overline{\alpha}=0$, that is $\hat{\omega}_{p+1}=d_0^{-1}d_1\overline{\alpha}$ with $\omega_{p+1}=z_{p+1}+d_1\beta_{p}$.
\end{itemize}
\subsubsection*{r=3} Using Definition \ref{Z and B defined} for $r=3$, we have
\begin{align*}
    \alpha\in Z_3^{p,\bullet}\ \Longleftrightarrow&\ d_0\alpha=0\text{ and there exist }z_{p+i}\in \C_{p+i,\bullet}\text{ with }i=1,2\text{ such that}\\ &\ d_1\alpha=d_0z_{p+1}\text{ and }d_2\alpha=d_1z_{p+1}+d_0z_{p+2}\,.
\end{align*}
From the previous steps, we already know that
\begin{itemize}
    \item $\alpha=d_0\beta_p+\overline{\alpha}$ for some $\beta_p\in\C_{p,\bullet}$;
    \item $d_c^1\overline{\alpha}=0$;
    \item $z_{p+1}=\omega_{p+1}-d_1\beta_p=\check{\omega}_{p+1}+\overline{\omega}_{p+1}+\hat{\omega}_{p+1}-d_1\beta_p=d_0\beta_{p+1}+\overline{\omega}_{p+1}+d_0^{-1}d_1\overline{\alpha}-d_1\beta_p$ for some $\beta_{p+1}\in\C_{p+1,\bullet}$
\end{itemize}
The additional condition $d_2\alpha=d_1z_{p+1}+d_0z_{p+2}$ then means that there exists some $z_{p+2}\in\C_{p+2,\bullet}$ such that
\begin{align*}
    d_2&(d_0\beta_p+\overline{\alpha})=d_1(d_0\beta_{p+1}+\overline{\omega}_{p+1}+d_0^{-1}d_1\overline{\alpha}-d_1\beta_p)+d_0z_{p+2}\\
    &\Leftrightarrow\ d_2\overline{\alpha}-d_1d_0^{-1}d_1\overline{\alpha}-d_1\overline{\omega}_{p+1}=-d_2d_0\beta_p+d_1d_0\beta_{p+1}-d_1^2\beta_p+d_0z_{p+2}\\
    &\Leftrightarrow\ \partial_2\overline{\alpha}-d_1\overline{\omega}_{p+1}=d_0d_2\beta_p+d_1^2\beta_p+d_1d_0\beta_{p+1}-d_1^2\beta_p+d_0z_{p+2}\\
    &\Leftrightarrow\ \partial_2\overline{\alpha}-d_1\overline{\omega}_{p+1}=d_0(d_2\beta_p-d_1\beta_{p+1}+z_{p+2})\,.
\end{align*}
Again, by Lemma \ref{d0 of partialr},  $\partial_2\overline{\alpha}-d_1\overline{\omega}_{p+1}\in\ker d_0$, since
\begin{align*}
    d_0(\partial_2\overline{\alpha}-d_1\overline{\omega}_{p+1})=&-d_1\left(d_1\overline{\alpha}-d_0d_0^{-1}d_1\overline{\alpha}\right)=-d_1\underbrace{d_c^1\overline{\alpha}}_{=0}=0\,.
\end{align*}
Just like before, we have
\begin{align*}
    \partial_2\overline{\alpha}-d_1\overline{\omega}_{p+1}=&\partial_2\overline{\alpha}-d_1\overline{\omega}_{p+1}-d_0d_0^{-1}\left(\partial_2\overline{\alpha}-d_1\overline{\omega}_{p+1}\right)+d_0d_0^{-1}\left(\partial_2\overline{\alpha}-d_1\overline{\omega}_{p+1}\right)\\=&d_c^2\overline{\alpha}-d_c^1\overline{\omega}_{p+1}+d_0d_0^{-1}\left(\partial_2\overline{\alpha}-d_1\overline{\omega}_{p+1}\right)=d_0\left(d_2\beta_p-d_1\beta_{p+1}+z_{p+2}\right)
\end{align*}
so that
    \begin{align*}
        \underbrace{d_c^2\overline{\alpha}-d_c^1\overline{\omega}_{p+1}}_{\in\ker\Box_{0}}=d_0\left[\omega_{p+2}-d_0^{-1}(\partial_2\overline{\alpha}-d_1\overline{\omega}_{p+1})\right]
    \end{align*}
    where we introduced $\omega_{p+2}:=d_2\beta_p-d_1\beta_{p+1}+z_{p+2}$. Reasoning like in the previous step, we get the equalities
    \begin{itemize}
        \item $d_c^2\overline{\alpha}-d_c^1\overline{\omega}_{p+1}=0$, and
        \item $\hat{\omega}_{p+2}=d_0^{-1}(\partial_2\overline{\alpha}-d_1\overline{\omega}_{p+1})$.
    \end{itemize}
    Therefore, the condition $\alpha\in Z_3^{p,\bullet}$ can be rephrased as
    \begin{itemize}
        \item $\alpha=d_0\beta_p+\overline{\alpha}$ for some $\beta_p\in\C_{p,\bullet}$;
        \item $d_c^1\overline{\alpha}=0$;
        \item $d_c^2\overline{\alpha}-d_c^1\overline{\omega}_{p+1}=0$ for some $\overline{\omega}_{p+1}\in\ker\Box_0\cap\C_{p+1,\bullet}$.
    \end{itemize}
 \subsubsection*{r=4} Using Definition \ref{Z and B defined} for $r=4$, we have
 \begin{align*}
         \alpha\in Z_4^{p,\bullet}\ \Longleftrightarrow&\ d_0\alpha=0\text{ and there exist }z_{p+i}\in\C_{p+i,\bullet}\text{ with }i=1,2,3\text{ such that}\\     &\ d_1\alpha=d_0z_{p+1}\,,\ d_2\alpha=d_1z_{p+1}+d_0z_{p+2}\,,\ d_3\alpha=d_2z_{p+1}+d_1z_{p+2}+d_0z_{p+3}\,.
 \end{align*}
     From the previous steps, we already know that
     \begin{itemize}
         \item $\alpha=d_0\beta_p+\overline{\alpha}$ for some $\beta_p\in\C_{p,\bullet}$;
         \item $d_c^1\overline{\alpha}=0$ and $z_{p+1}=d_0\beta_{p+1}+\overline{\omega}_{p+1}+d_0^{-1}d_1\overline{\alpha}-d_1\beta_p$ for some $\beta_{p+1}\in\C_{p+1,\bullet}$ and $\overline{\omega}_{p+1}\in\C_{p+1,\bullet}\cap\ker\Box_0$;
         \item $d_c^{2}\overline{\alpha}-d_c^1\overline{\omega}_{p+1}=0$, and
         \item $z_{p+2}=\omega_{p+2}+d_1\beta_{p+1}-d_2\beta_p=\check{\omega}_{p+2}+\overline{\omega}_{p+2}+\hat{\omega}_{p+2}+d_1\beta_{p+1}-d_2\beta_p=d_0\beta_{p+2}+\overline{\omega}_{p+2}+d_0^{-1}(\partial_2\overline{\alpha}-d_1\overline{\omega}_{p+1})+d_1\beta_{p+1}-d_2\beta_p$ for some $\beta_{p+2}\in\C_{p+2,\bullet}$ and $\overline{\omega}_{p+2}\in\C_{p+2,\bullet}\cap\ker\Box_0$.
     \end{itemize}
     The additional condition 
 \begin{align*}
d_3\alpha=d_2z_{p+1}+d_1z_{p+2}+d_0z_{p+3}\text{ for some }z_{p+3}\in\C_{p+3,\bullet}
     \end{align*}
     then reads
    \begin{align*}     d_3(d_0\beta_p+\overline{\alpha})=&d_2(d_0\beta_{p+1}+\overline{\omega}_{p+1}+d_0^{-1}d_1\overline{\alpha}-d_1\beta_p)+\\&+d_1\left[d_0\beta_{p+2}+\overline{\omega}_{p+2}+d_0^{-1}(\partial_2\overline{\alpha}-d_1\overline{\omega}_{p+1})+d_1\beta_{p+1}-d_2\beta_p\right]+d_0z_{p+3}\\
     \end{align*}
     Just like before, $\partial_3\overline{\alpha}-\partial_2\overline{\omega}_{p+1}-d_1\overline{\omega}_{p+2}\in\ker d_0$, since
    \begin{align*}  d_0\left(\partial_3\overline{\alpha}-\partial_2\overline{\omega}_{p+1}-d_1\overline{\omega}_{p+2}\right)=&-d_2d_c^1\overline{\alpha}-d_1\left(d_c^2\overline{\alpha}-d_c^1\overline{\omega}_{p+1}\right)=0
     \end{align*}
     and so we get
     \begin{itemize}
         \item $d_c^{3}\overline{\alpha}-d_c^2\overline{\omega}_{p+1}-d_c^1\overline{\omega}_{p+2}=0$ and
         \item $\hat{\omega}_{p+3}=d_0^{-1}\left(\partial_3\overline{\alpha}-\partial_2\overline{\omega}_{p+1}-d_1\overline{\omega}_{p+2}\right)$ where $\omega_{p+3}=d_3\beta_p-d_2\beta_{p+1}-d_1\beta_{p+2}+z_{p+3}$.
     \end{itemize}
     In other words, the condition $\alpha\in Z_4^{p,\bullet}$ can be rephrased as
     \begin{itemize}
         \item $\alpha=d_0\beta_p+\overline{\alpha}$ for some $\beta_p\in\C_{p,\bullet}$;
         \item $d_c^1\overline{\alpha}=0$ and $z_{p+1}=d_0\beta_{p+1}-d_1\beta_p+\overline{\omega}_{p+1}+d_0^{-1}d_1\overline{\alpha}$ for some $\beta_{p+1}\in\C_{p+1,\bullet}$ and $\overline{\omega}_{p+1}\in\C_{p+1,\bullet}\cap\ker\Box_0$;
         \item $d_c^2\overline{\alpha}-d_c^1\overline{\omega}_{p+1}=0$ and $z_{p+2}=d_0\beta_{p+2}+d_1\beta_{p+1}-d_2\beta_p+\overline{\omega}_{p+2}+d_0^{-1}\left(\partial_2\overline{\alpha}-d_1\overline{\omega}_{p+1}\right)$ for some $\beta_{p+2}\in\C_{p+2,\bullet}$ and $\overline{\omega}_{p+2}\in\C_{p+2,\bullet}\cap\ker\Box_0$;
         \item $d_c^3\overline{\alpha}-d_c^2\overline{\omega}_{p+1}-d_c^1\overline{\omega}_{p+2}=0$ and $z_{p+3}=d_0\beta_{p+3}+d_1\beta_{p+2}+d_2\beta_{p+1}-d_3\beta_p+\overline{\omega}_{p+3}+d_0^{-1}(\partial_3\overline{\alpha}-\partial_2\overline{\omega}_{p+1}-\partial_1\overline{\omega}_{p+2})$ for some $\beta_{p+3}\in\C_{p+3,\bullet}$ and $\overline{\omega}_{p+3}\in
         \C_{p+3,\bullet}\cap\ker\Box_0$
     \end{itemize}
         Expressed in terms of Rumin forms and their differentials, this means that $\alpha\in Z_4^{p,\bullet}$ if and only if
     \begin{itemize}
         \item $\alpha=d_0\beta_p+\overline{\alpha}$ for some $\beta_p\in\C_{p,\bullet}$;
         \item $d_c^1\overline{\alpha}=0$;
         \item $d_c^2\overline{\alpha}-d_c^1\overline{\omega}_{p+1}=0$ for some $\overline{\omega}_{p+1}\in\C_{p+1,\bullet}\cap\ker\Box_0$;
         \item $d_c^3\overline{\alpha}-d_c^2\overline{\omega}_{p+1}-d_c^1\overline{\omega}_{p+2}=0$ for some $\overline{\omega}_{p+2}\in\C_{p+2,\bullet}\cap\ker\Box_0$.
     \end{itemize}
    \begin{proposition}\label{prop expression Z_r}Let us prove by induction that for any $r\ge 2$, the condition $\alpha\in Z_{r}^{p,\bullet}$ is equivalent to saying that $\alpha=d_0\beta_p+\overline{\alpha}$ for some $\beta_p\in\C_{p,\bullet}$ and there exist $\overline{\omega}_{p+i}\in\C_{p+i,\bullet}\cap\ker\Box_0$ with $i=1,\ldots,r-2$ such that
    \begin{align*}
        d_c^i\overline{\alpha}=\sum_{j=1}^{i-1}d_c^{i-j}\overline{\omega}_{p+j}\ \text{ for each }i=1,\ldots,r-1
    \end{align*}
        
    \end{proposition}
    \begin{proof} 
To show the claim, it will be necessary to also show that for each $i=1,\ldots,r-1$
\begin{itemize}
    \item $\omega_{p+i}=z_{p+i}+d_i\beta_p-\sum_{j=1}^{i-1}d_{i-j}\beta_{p+j}$, where
    \item $d_0\beta_{p+i}=\check{\omega}_{p+i}$ and $\hat{\omega}_{p+i}=d_0^{-1}\left(\partial_i\overline{\alpha}-\sum_{j=1}^{i-1}\partial_{i-j}\overline{\omega}_{
   p+j} \right)$, and
   \item $\partial_i\overline{\alpha}-\sum_{j=1}^{i-1}\partial_{i-j}\overline{\omega}_{p+j}\in\ker d_0$.
\end{itemize}
Let us show this by induction. Given the previous explicit cases where $r\le 4$, we are left to show that if these formulae hold for $\alpha\in Z_{r}^{p,\bullet}$, then they also hold for $Z_{r+1}^{p,\bullet}$.

Following Definition \ref{Z and B defined}, the condition $\alpha\in Z_{r+1}^{p,\bullet}$ simply reads  as $\alpha\in Z_r^{p,\bullet}$ and additionally there exists an element $z_{p+r}\in\C_{p+r,\bullet}$ such that
\begin{align*}
    d_r\alpha=d_{r-1}z_{p+1}+d_{r-2}z_{p+2}+\cdots+d_1z_{p+r-1}+d_0z_{p+r}
\end{align*}
where, by the inductive hypothesis,
\begin{itemize}
    \item $\alpha=d_0\beta_p+\overline{\alpha}$ for some $\beta_p\in\C_{p,\bullet}$, and
    \item $z_{p+i}=\omega_{p+i}-d_i\beta_p+\sum_{j=1}^{i-1}d_{i-j}\beta_{p+j}=d_0\beta_{p+i}+\overline{\omega}_{p+i}+d_0^{-1}\left(\partial_i\overline{\alpha}-\sum_{j=1}^{i-1}\partial_{i-j}\overline{\omega}_{p+j}\right)-d_i\beta_p+\sum_{j=1}^{i-1}d_{i-j}\beta_{p+j}$ for each $i=1,\ldots,r-1$.
\end{itemize}
Therefore, we get
\begin{align*}
    d_r\alpha=&d_r(d_0\beta_p+\overline{\alpha})=\sum_{i=1}^{r-1}d_{r-i}z_{p+i}+d_0z_{p+r}=\sum_{i=1}^{r-1}d_{r-i}d_0\beta_{p+i}+\sum_{i=1}^{r-1}d_{r-i}\overline{\omega}_{p+i}+\sum_{i=1}^{r-1}d_{r-i}d_0^{-1}\partial_i\overline{\alpha}+\\&-\sum_{i=1}^{r-1}d_{r-i}d_0^{-1}\sum_{j=1}^{i-1}\partial_{i-j}\overline{\omega}_{p+j}-\sum_{i=1}^{r-1}d_{r-i}d_i\beta_p+\sum_{i=1}^{r-1}d_{r-i}\sum_{j=1}^{i-1}d_{i-j}\beta_{p+j}+d_0z_{p+r}
\end{align*}
\begin{remark}Using equalities \eqref{d0dr} and by re-labelling the sums, we get that the addends
\begin{align*}
    \sum_{i=1}^{r-1}d_{r-i}d_0\beta_{p+i}+\sum_{i=1}^{r-1}\sum_{j=1}^{i-1}d_{r-i}d_{i-j}\beta_{p+j}\ \text{ and }\sum_{i=1}^{r-1}d_{r-i}d_0^{-1}\sum_{j=1}^{i-1}\partial_{i-j}\overline{\omega}_{p+j}
\end{align*}
get simplified. Indeed
\begin{align*}
    \sum_{i=1}^{r-1}d_{r-i}&d_0\beta_{p+i}+\sum_{i=1}^{r-1}\sum_{j=1}^{i-1}d_{r-i}d_{i-j}\beta_{p+j}=d_1d_0\beta_{p+r-1}+\sum_{i=1}^{r-2}d_{r-i}d_0\beta_{p+i}+\sum_{j=1}^{r-2}\sum_{i=j+1}^{r-1}d_{r-i}d_{i-j}\beta_{p+j}\\ \underbrace{=}_{i-j=k}&d_1d_0\beta_{p+r-1}+\sum_{i=1}^{r-2}d_{r-i}d_0\beta_{p+i}+\sum_{j=1}^{r-2}\sum_{k=1}^{r-1-j}d_{r-k-j}d_k\beta_{p+j}\\=&d_1d_0\beta_{p+r-1}+\sum_{i=1}^{r-2}\left(d_{r-i}d_0+\sum_{k=1}^{r-i-1}d_{r-i-k}d_k\right)\beta_{p+i}=-d_0d_1\beta_{p+r-1}-\sum_{i=1}^{r-2}d_0d_{r-i}\beta_{p+i}\\=&-d_0\left(d_1\beta_{p+r-1}+\sum_{i=1}^{r-2}d_{r-i}\beta_{p+i}\right)=-d_0\sum_{i=1}^{r-1}d_{r-i}\beta_{p+i}
\end{align*}
and
\begin{align*}
    \sum_{i=1}^{r-1}d_{r-i}d_0^{-1}\sum_{j=1}^{i-1}\partial_{i-j}\overline{\omega}_{p+j}=&\sum_{j=1}^{r-2}\sum_{i=j+1}^{r-1}d_{r-i}d_0^{-1}\partial_{i-j}\overline{\omega}_{p+j}\underbrace{=}_{i-j=k}\sum_{j=1}^{r-2}\sum_{k=1}^{r-1-j}d_{r-(k+j)}d_0^{-1}\partial_k\overline{\omega}_{p+j}\\=&\sum_{i=1}^{r-2}\sum_{j=1}^{r-i-1}d_{(r-i)-j}d_0^{-1}\partial_j\overline{\omega}_{p+i}\,.
\end{align*}
    
\end{remark}
Using these simplifications, we get 
\begin{align*}
    d_r\overline{\alpha}-\sum_{i=1}^{r-1}d_{r-i}d_0^{-1}\partial_i\overline{\alpha}-\sum_{i=1}^{r-1}&d_{r-i}\overline{\omega}_{p+i}+\sum_{i=1}^{r-2}\sum_{j=1}^{r-i-1}d_{(r-i)-j}d_0^{-1}\partial_j\overline{\omega}_{p+i}=\\=&-d_0\sum_{i=1}^{r-1}d_{r-i}\beta_{p+i}-d_rd_0\beta_p-\sum_{i=1}^{r-1}d_{r-i}d_i\beta_p+d_0z_{p+r}\\
    \Leftrightarrow\ \partial_r\overline{\alpha}-\sum_{i=1}^{r-1}\partial_{r-i}\overline{\omega}_{p+i}=d_0&\left(d_r\beta_p-\sum_{i=1}^{r-1}d_{r-i}\beta_{p+i}+z_{p+r}\right)\,.
\end{align*}
The claim that $\partial_r\overline{\alpha}-\sum_{i=1}^{r-1}\partial_{r-i}\overline{\omega}_{p+i}\in\ker d_0$ follows from Lemma \ref{d0 of partialr} together with the inductive hypothesis, since
\begin{align*}
    d_0\Bigg( \partial_r\overline{\alpha}&-\sum_{i=1}^{r-1}\partial_{r-i}\overline{\omega}_{p+i} \Bigg) =-\sum_{j=1}^{r-1}d_j\left(\partial_{r-j}-d_0d_0^{-1}\partial_{r-j}\right)\overline{\alpha}+\sum_{i=1}^{r-1}\sum_{j=1}^{r-i-1}d_j\left(\partial_{r-i-j}-d_0d_0^{-1}\partial_{r-i-j}\right)\overline{\omega}_{p+i}\\=&-d_{r-1}\left(\partial_1-d_0d_0^{-1}\partial_1\right)\overline{\alpha}-\sum_{j=1}^{r-2}d_j\left(\partial_{r-j}-d_0d_0^{-1}\partial_{r-j}\right)\overline{\alpha}+\sum_{j=1}^{r-2}\sum_{i=1}^{r-j-1}d_j\left(\partial_{r-i-j}-d_0d_0^{-1}\partial_{r-i-j}\right)\overline{\omega}_{p+i}\\=&-d_{r-1}d_c^1\overline{\alpha}+\sum_{j=1}^{r-2}d_j\left(d_c^{r-j}\overline{\alpha}-\sum_{i=1}^{r-j-1}d_c^{(r-j)-i}\overline{\omega}_{p+i}\right)=0\,.
\end{align*}
Therefore,
\begin{align*}
    d_c^{r}\overline{\alpha}-\sum_{i=1}^{r-1}d_c^{r-i}\overline{\omega}_{p+i}+d_0d_0^{-1}\left(\partial_r\overline{\alpha}-\sum_{i=1}^{r-1}\partial_{r-i}\overline{\omega}_{p+i}\right)=d_0\left(d_r\beta_p-\sum_{i=1}^{r-1}d_{r-i}\beta_{p+i}+z_{p+r}\right)=d_0\omega_{p+r}
\end{align*}
if we impose $\omega_{p+r}=d_r\beta_p-\sum_{i=1}^{r-1}d_{r-i}\beta_{p+i}+z_{p+r}$.

To conclude, we have shown that 
\begin{itemize}
    \item $d_c^r\overline{\alpha}-\sum_{i=1}^{r-1}d_c^{r-i}\overline{\omega}_{p+i}=0$ by the Hodge decomposition \ref{eq hodge decomp};
    \item $\hat{\omega}_{p+r}=d_0^{-1}\left(\partial_r\overline{\alpha}-\sum_{i=1}^{r-1}\partial_{r-i}\overline{\omega}_{p+i}\right)$, and 
    \item $\check{\omega}_{p+r}=d_0\beta_{p+r}$ for some $\beta_{p+r}\in\C_{p+r,\bullet}$.
\end{itemize}
\end{proof}
\subsection{Expressing $B_r^{p,\bullet}$ in terms of Rumin forms and the Rumin differential}
The calculations in this case can be streamlined using the results of Proposition \ref{prop expression Z_r}.
\subsubsection*{r=1} Using Definition \ref{Z and B defined} for $r=1$, we have
\begin{align*}
    \alpha\in B_1^{p,\bullet}\ \Longleftrightarrow\ \exists\, c_p\in\C_{p,\bullet}\text{ such that }\alpha=d_0c_p\,.
\end{align*}
In other words, $\alpha$ belongs to $B_1^{p,\bullet}$ if it is an element in $\mathrm{Im}\,d_0$.
\subsubsection*{r=2} Using Definition \ref{Z and B defined} for $r=2$, we have
\begin{align*}
    \alpha\in B_2^{p,\bullet}\ \Longleftrightarrow&\ \exists\,c_{p-i}\in \C_{p-i,\bullet}\text{ with }i=0,1\text{ such that }\\
    &\ d_0c_{p-1}=0\text{ and }d_1c_{p-1}+d_0c_p=\alpha\,.
\end{align*}
In other words, we are saying $c_{p-1}\in Z_1^{p-1,\bullet}$ with $c_{p-1}=d_0\beta_{p-1}+\overline{c}_{p-1}$ for some $\beta_{p-1}\in\C_{p-1,\bullet}$. In addition,
\begin{align*}
    \alpha=d_1c_{p-1}+d_0c_{p}=d_1(d_0\beta_{p-1}+\overline{c}_{p-1})+d_0c_p=d_1\overline{c}_{p-1}-d_0d_1\beta_{p-1}+d_0c_{p}\,.
\end{align*}
Since $d_1\overline{c}_{p-1}\in\ker d_0$, we have
 $$\alpha=d_c^1\overline{c}_{p-1}+d_0d_0^{-1}d_1\overline{c}_{p-1}-d_0d_1\beta_{p-1}+d_0c_p\,.$$
 This in particular implies that $\alpha\in\ker\Box_0\oplus\mathrm{Im\,}d_0=\ker d_0$.
\subsubsection*{r=3}
Using Definition \ref{Z and B defined} for $r=3$, we have
\begin{align*}
    \alpha\in B_3^{p,\bullet}\ \Longleftrightarrow&\ \exists\,c_{p-i}\in\C_{p-i,\bullet}\text{ with }i=0,1,2\text{ such that}\\
    &\ d_0c_{p-2}=0\ ,\ d_1c_{p-2}+d_0c_{p-1}=0\text{ and }d_2c_{p-2}+d_1c_{p-1}+d_0c_p=\alpha\,.
\end{align*}
Again, the first equations imply that $c_{p-2}\in Z_{2}^{p-2,\bullet}$ with $c_{p-2}=d_0\beta_{p-2}+\overline{c}_{p-2}$ for some $\beta_{p-2}\in\C_{p-2,\bullet}$ and
\begin{align*}
    d_1(d_0\beta_{p-2}+\overline{c}_{p-2})+d_0c_{p-1}=0\ \Longleftrightarrow d_1\overline{c}_{p-2}=d_0d_1\beta_{p-2}-d_0c_{p-1}\,.
\end{align*}
By imposing $\omega_{p-1}=d_1\beta_{p-2}-c_{p-1}$ and repeating the same reasoning as before, we have
\begin{itemize}
    \item $\partial_1\overline{c}_{p-2}\in\ker d_0$;
    \item $d_c^1\overline{c}_{p-2}=0$;
    \item $\hat{\omega}_{p-1}=d_0^{-1}d_1\overline{c}_{p-2}$ and $\check{\omega}_{p-1}=d_0\beta_{p-1}$ for some $\beta_{p-1}\in\C_{p-1,\bullet}$.
\end{itemize}
Furthermore, the final equation reads
\begin{align*}
    \alpha=&d_2(d_0\beta_{p-2}+\overline{c}_{p-2})+d_1(d_1\beta_{p-2}-\omega_{p-1})+d_0c_p\\=&d_2d_0\beta_{p-2}+d_2\overline{c}_{p-2}+d_1^2\beta_{p-2}-d_1d_0\beta_{p-1}-d_1\overline{\omega}_{p-1}-d_1d_0^{-1}d_1\overline{c}_{p-2}+d_0c_p\\=&\partial_2\overline{c}_{p-2}-d_1\overline{\omega}_{p-1}-d_0d_2\beta_{p-2}+d_0d_1\beta_{p-1}+d_0c_p\\
    =&d_c^2\overline{c}_{p-2}-d_c^1\overline{\omega}_{p-1}+d_0d_0^{-1}\left(\partial_2\overline{c}_{p-2}-d_1\overline{\omega}_{p-1}\right)-d_0d_2\beta_{p-2}+d_0d_1\beta_{p-1}+d_0c_p\,.
\end{align*}
Notice that here we are using the fact that $\partial_2\overline{c}_{p-2}-\partial_1\overline{\omega}_{p-1}\in\ker d_0$, which also implies that $\alpha\in\ker\Box_0\oplus\mathrm{Im}\,d_0=\ker d_0$.

\begin{proposition}\label{prop expression B_r}Let us prove that for any $r\ge 2$, the condition $\alpha\in B_r^{p,\bullet}$ is equivalent to saying that there exists $c_{p-r+1}\in Z_{r-1}^{p-(r-1),\bullet}$, and so $c_{p-r+1}=d_0\beta_{p-r+1}+\overline{c}_{p-r+1}$ for some $\beta_{p-r+1}\in\C_{p-r+1,\bullet}$ as well as $\overline{\omega}_{p-r+i}\in\C_{p-r+i,\bullet}\cap\ker\Box_0$ with $i=2,\ldots,r-1$ such that
\begin{align*}
    d_c^i\overline{c}_{p-r+1}=\sum_{j=1}^{i-1}d_c^{i-j}\overline{\omega}_{p-r+1+j}\ \text{ for each }i=1,\ldots,r-2
\end{align*}
and
\begin{align*}
    \alpha=&d_c^{r-1}\overline{c}_{p-r+1}-\sum_{i=1}^{r-2}d_c^{r-1-i}\overline{\omega}_{p-r+1+i}+d_0d_0^{-1}\left(\partial_{r-1}\overline{c}_{p-r+1}-\sum_{i=1}^{r-2}\partial_{r-1-i}\overline{\omega}_{p-r+1+i}\right)+\\&-d_0d_{r-1}\beta_{p-r+1}+d_0\sum_{i=1}^{r-2}d_i\beta_{p-i}+d_0c_p\ \text{ for some }\beta_{p-i}\in\C_{p-i,\bullet} \text{ with }i=1,\ldots,r-2\,.
\end{align*}
    
\end{proposition}
\begin{proof}
    By Definition \ref{Z and B defined}, if $\alpha\in B_r^{p,\bullet}$ then there exists $c_{p-r+1}\in\C_{p-r+1,\bullet}$ that belongs to $Z_{r-1}^{p-r+1,\bullet}$. So by what we proved in Proposition \ref{prop expression Z_r}, we have that $c_{p-r+1}=d_0\beta_{p-r+1}+\overline{c}_{p-r+1}$ for some $\beta_{p-r+1}\in\C_{p-r+1,\bullet}$. Moreover, there exist $\beta_{p-r+1+i}\in\C_{p-r+1+i,\bullet}$ and $\overline{\omega}_{p-r+1+i}\in\C_{p-r+1+i,\bullet}\cap\ker\Box_0$ for $i=1,\ldots,r-3$ such that
    \begin{align*}
        d_c^i\overline{c}_{p-r+1}=\sum_{j=1}^{i-1}d_c^{i-j}\overline{\omega}_{p-r+1+j}\ \text{ for each }i=1,\ldots,r-2\,.
    \end{align*}
    Notice that for $i= 1,\ldots,r-3$ (using the same notation used to define the $B_r^{p,\bullet}$ in Definition \ref{Z and B defined})
    \begin{itemize}
        \item $\omega_{p-r+1+i}=-c_{p-r+1+i}+d_i\beta_{p-r+1}-\sum_{j=1}^{i-1}d_{i-j}\beta_{p-r+1+j}$, where
        \item $d_0\beta_{p-r+1+i}=\check{\omega}_{p-r+1+i}$ and $\hat{\omega}_{p-r+1+i}=d_0^{-1}\left(\partial_i\overline{c}_{p-r+1}-\sum_{j=1}^{i-1}\partial_{i-j}\overline{\omega}_{p-r+1+j}\right)$ and
        \item $\partial_i\overline{c}_{p-r+1}-\sum_{j=1}^{i-1}\partial_{i-j}\overline{\omega}_{p-r+1+j}\in\ker d_0$.
    \end{itemize}
    We are left to study the final equation where
    \begin{align*}
        \alpha=\sum_{i=0}^{r-2}d_{r-1-i}c_{p-r+1+i}+d_0c_{p}\ \text{ for some }c_p\in\C_{p,\bullet}\,.
    \end{align*}
    Using the fact that $c_{p-r+1}\in Z_{r-1}^{p-r+1,\bullet}$ as stated above, we have
    \begin{align*}
        \sum_{i=0}^{r-2}d_{r-1-i}c_{p-r+1+i}=&d_{r-1}d_0\beta_{p-r+1}+d_{r-1}\overline{c}_{p-r+1}-\sum_{i=1}^{r-2}d_{r-1-i}\omega_{p-r+1+i}+\sum_{i=1}^{r-2}d_{r-1-i}d_i\beta_{p-r+1}+\\&-\sum_{i=1}^{r-2}d_{r-1-i}\sum_{j=1}^{i-1}d_{i-j}\beta_{p-r+1+j}\\=&d_{r-1}d_0\beta_{p-r+1}+d_{r-1}\overline{c}_{p-r+1}-\sum_{i=1}^{r-2}d_{r-1-i}d_0\beta_{p-r+1+i}-\sum_{i=1}^{r-2}d_{r-1-i}\overline{\omega}_{p-r+1+i}+\\&-\sum_{i=1}^{r-2}d_{r-1-i}d_0^{-1}\left(\partial_i\overline{c}_{p-r+1}-\sum_{j=1}^{i-1}\partial_{i-j}\overline{\omega}_{p-r+1+j}\right)+\\&+\sum_{i=1}^{r-2}d_{r-1-i}d_i\beta_{p-r+1}-\sum_{i=1}^{r-2}d_{r-1-i}\sum_{j=1}^{i-1}d_{i-j}\beta_{p-r+1+j}\\=&\partial_{r-1}\overline{c}_{p-r+1}-\sum_{i=1}^{r-2}\partial_{r-1-i}\overline{\omega}_{p-r+1+i}-d_0d_{r-1}\beta_{p-r+1}+d_0\sum_{i=1}^{r-2}d_{r-1-i}\beta_{p-r+1+i}\\=&\partial_{r-1}\overline{c}_{p-r+1}-\sum_{i=1}^{r-2}\partial_{r-1-i}\overline{\omega}_{p-r+1+i}-d_0d_{r-1}\beta_{p-r+1}+d_0\sum_{i=1}^{r-2}d_{i}\beta_{p-i}
    \end{align*}
    and hence we get the claim.
\end{proof}
\subsection{Expressing the differentials $\Delta_r$ of the spectral sequences in terms of Rumin forms and the Rumin differential}
Using the $Z_r^{p,\bullet}$ and $B_r^{p,\bullet}$ graded submodules, it is possible to have an explicit formulation of the differentials arising at each page of the spectral sequence.
\begin{proposition}[Theorem 2.10 in \cite{livernet2020spectral}]\label{theorem Delta_r}
    The $r^{th}$ differential of the spectral sequence corresponds to the map
    \begin{align*}
        \Delta_r\colon Z_r^{p,\bullet}/B_{r}^{p,\bullet}\longrightarrow Z_r^{p+r,\bullet}/B_r^{p+r,\bullet}\\ \Delta_r\left([x]\right)=\left[d_rx-\sum_{i=1}^{r-1}d_iz_{p+r-i}\right]
    \end{align*}
    where $x\in Z_r^{p,\bullet}$ and the family $\lbrace z_{p+j}\rbrace_{1\le j\le r-1}$ satisfies the equations of Definition \ref{Z and B defined}
\end{proposition}
Before proving the claim in full generality, let us first see a few cases explicitly. To do so, we will follow what we have already shown in the previous subsections.
\subsubsection*{r=1} In this case, we have $\alpha\in Z_1^{p,\bullet}$, that is $\alpha=d_0\beta_p+\overline{\alpha}$ for some $\beta_p\in\C_{p,\bullet}$. The action of $\Delta_1$ then takes the form
\begin{align*}
    \Delta_1\alpha=&d_1\alpha=d_1d_0\beta_p+d_1\overline{\alpha}=d_c^1\overline{\alpha}+d_0d_0^{-1}d_1\overline{\alpha}-d_0d_1\beta_p=d_c^1\overline{\alpha}+d_0\left(d_0^{-1}d_1\overline{\alpha}-d_1\beta_p\right)\,.
\end{align*}
Since an element is in $B_1^{p+1,\bullet}$ if it belongs to the $\mathrm{Im}\,d_0$, then
\begin{align*}
    \left[d_c^1\overline{\alpha}+d_0\left(d_0^{-1}d_1\overline{\alpha}-d_1\beta_p\right)\right]=\left[d_c^1\overline{\alpha}\right]\,.
\end{align*}
\subsubsection*{r=2} In this case, we have $\alpha\in Z_2^{p,\bullet}$ if there exist $\beta_{p}\in\C_{p,\bullet}$ and $z_{p+1}\in\C_{p+1,\bullet}$ such that
\begin{align*}
    \alpha=d_0\beta_p+\overline{\alpha}\ \text{ and }\ d_1\alpha-d_0z_{p+1}=0\,.
\end{align*}
By what we have seen previously, if we impose $\omega_{p+1}=z_{p+1}+d_1\beta_p$, then
\begin{itemize}
    \item $d_c^1\overline{\alpha}=0$;
    \item $\check{\omega}_{p+1}=d_0\beta_{p+1}$ for some $\beta_{p+1}\in\C_{p+1,\bullet}$, and
    \item $\hat{\omega}_{p+1}=d_0^{-1}d_1\overline{\alpha}$
\end{itemize}
so that
\begin{align*}
    d_2\alpha-d_1z_{p+1}=&d_2d_0\beta_p+d_2\overline{\alpha}-d_1d_0\beta_{p+1}-d_1\overline{\omega}_{p+1}-d_1d_0^{-1}d_1\overline{\alpha}+d_1^2\beta_p\\=&\partial_2\overline{\alpha}-d_1\overline{\omega}_{p+1}+d_0d_1\beta_{p+1}-d_0d_2\beta_p\\=&d_c^2\overline{\alpha}-d_c^1\overline{\omega}_{p+1}+d_0d_0^{-1}\left(\partial_2\overline{\alpha}-d_1\overline{\omega}_{p+1}\right)+d_0d_1\beta_{p+1}-d_0d_2\beta_p\,.
\end{align*}
Notice that an element $x\in B_2^{p+2,\bullet}$, if there exists $c_{p+1}\in Z_{1}^{p+1,\bullet}$ with $c_{p+1}=d_0\beta_{p+1}+\overline{c}_{p+1}$ for some $\beta_{p+1}\in\C_{p+1,\bullet}$ and
\begin{align*}
    x=d_c^1\overline{c}_{p+1}+d_0d_0^{-1}d_1\overline{c}_{p+1}-d_0d_1\beta_{p+1}+d_0c_{p+2}\ \text{ for some }c_{p+2}\in\C_{p+2,\bullet}\,.
\end{align*}
Therefore, if we choose
\begin{itemize}
    \item $c_{p+1}=-d_0\beta_{p+1}-\overline{\omega}_{p+1}$, and
    \item $c_{p+2}=d_0^{-1}\partial_2\overline{\alpha}-d_2\beta_p$
\end{itemize}
we have
\begin{align*}
    \Delta_2\left(\left[\alpha\right]\right)=\left[d_c^2\overline{\alpha}-\left(d_c^1\overline{\omega}_{p+1}+d_0d_0^{-1}d_1\overline{\omega}_{p+1}-d_0d_1\beta_{p+1}\right)+d_0\left(d_0^{-1}\partial_2\overline{\alpha}-d_2\beta_p\right)\right]=\left[d_c^2\overline{\alpha}\right]\,.
\end{align*}
\subsubsection*{r=3} In this case, we have $\alpha\in Z_3^{p,\bullet}$ if there exist $\beta_p\in\C_{p,\bullet}$ and $z_{p+i}\in\C_{p+i,\bullet}$ for $i=1,2$ such that
\begin{align*}
    \alpha=d_0\beta_p+\overline{\alpha}\ ,\ d_1\alpha-d_0z_{p+1}=0\ ,\ d_2\alpha-d_1z_{p+1}-d_0z_{p+2}=0\,.
\end{align*}
If we impose $\omega_{p+1}=z_{p+1}+d_1\beta_p$ and $\omega_{p+2}=d_2\beta_p-d_1\beta_{p+1}+z_{p+2}$, we have
\begin{itemize}
    \item $d_c^1\overline{\alpha}=d_c^2\overline{\alpha}-d_c^1\overline{\omega}_{p+1}=0$,
    \item $\check{\omega}_{p+1}=d_0\beta_{p+1}$ and $\check{\omega}_{p+2}=d_0\beta_{p+2}$ for some $\beta_{p+i}\in\C_{p+i,\bullet}$ with $i=1,2$;
    \item $\hat{\omega}_{p+1}=d_0^{-1}d_1\overline{\alpha}$ and $\hat{\omega}_{p+2}=d_0^{-1}\left(\partial_2\overline{\alpha}-d_1\overline{\omega}_{p+1}\right)$
\end{itemize}
we get
\begin{align*}
    d_3\alpha-d_2z_{p+1}-d_1z_{p+2}=&d_3d_0\beta_p+d_3\overline{\alpha}-d_2d_0\beta_{p+1}-d_2\overline{\omega}_{p+1}-d_2d_0^{-1}d_1\overline{\alpha}+d_2d_1\beta_p+\\&-d_1d_0\beta_{p+2}-d_1\overline{\omega}_{p+2}-d_1d_0^{-1}\left(\partial_2\overline{\alpha}-d_1\overline{\omega}_{p+1}\right)+d_1d_2\beta_p-d_1^2\beta_{p+1}\\=&\partial_3\overline{\alpha}-\partial_2\overline{\omega}_{p+1}-d_1\overline{\omega}_{p+2}-d_0d_3\beta_p+d_0d_2\beta_{p+1}+d_0d_1\beta_{p+2}\\=&d_c^3\overline{\alpha}-d_c^2\overline{\omega}_{p+1}-d_c^1\overline{\omega}_{p+2}+d_0d_0^{-1}\left(\partial_3\overline{\alpha}-\partial_2\overline{\omega}_{p+1}-d_1\overline{\omega}_{p+2}\right)+\\&-d_0d_3\beta_p+d_0d_2\beta_{p+1}+d_0d_1\beta_{p+2}\,.
\end{align*}
Notice that an element $x\in B_3^{p+3,\bullet}$ if there exist $c_{p+1}\in Z_2^{p+1,\bullet}$ with $c_{p+1}=d_0\beta_{p+1}+\overline{c}_{p+1}$ for some $\beta_{p+1}\in\C_{p+1,\bullet}$ such that $d_c^1\overline{c}_{p+1}=0$ and
\begin{align*}
    x=d_c^2\overline{c}_{p+1}-d_c^1\overline{\omega}_{p+2}+d_0d_0^{-1}\left(\partial_2\overline{c}_{p+1}-d_1\overline{\omega}_{p+2}\right)-d_0d_2\beta_{p+1}+d_0d_1\beta_{p+2}+d_0c_{p+3}
\end{align*}
for some $\omega_{p+2}\in\C_{p+2,\bullet}$ such that $\check{\omega}_{p+2}=d_0\beta_{p+2}$ and some $c_{p+3}\in\C_{p+3,\bullet}$.
Therefore, if we choose
\begin{itemize}
    \item $c_{p+1}=-d_0\beta_{p+1}$, i.e. we are taking $\overline{c}_{p+1}=0$,
    \item $c_{p+2}=-d_1\beta_{p+1}+d_0\beta_{p+2}+\overline{\omega}_{p+2}$, and
    \item $c_{p+3}=d_0^{-1}\left(\partial_3\overline{\alpha}-\partial_2\overline{\omega}_{p+1}\right)-d_3\beta_p$
\end{itemize}
we have
\begin{align*}
    \Delta_3\left([\alpha]\right)=&\big[d_c^3\overline{\alpha}-d_c^2\overline{\omega}_{p+1}-d_c^1\overline{\omega}_{p+2}-d_0d_0^{-1}\left(\partial_2\overline{\omega}_{p+1}+d_1\overline{\omega}_{p+2}\right)+d_0d_2\beta_{p+1}+d_0d_1\beta_{p+2}\\&+d_0d_0^{-1}\partial_3\overline{\alpha}-d_0d_3\beta_p\big]=\left[d_c^3\overline{\alpha}-d_c^2\overline{\omega}_{p+1}\right]\,.
\end{align*}
Most importantly, in general, we are not able to take $c_{p+1}=-d_0\beta_{p+1}-\overline{\omega}_{p+1}$ and $c_{p+3}=d_0^{-1}\partial_3\overline{\alpha}-d_3\beta_p$, even though this would give us
\begin{align*}
    x=-d_c^2\overline{\omega}_{p+1}-d_c^1\overline{\omega}_{p+2}+d_0d_0^{-1}\left(\partial_3\overline{\alpha}-\partial_2\overline{\omega}_{p+1}-d_1\overline{\omega}_{p+2}\right)+d_0\left(d_2\beta_{p+1}+d_1\beta_{p+2}-d_3\beta_p\right)
\end{align*}
and further simplify the expression of $\Delta_3$. This is because in general $d_c^1\overline{\omega}_{p+1}$ will not vanish. Actually, the condition $d_c^1\overline{\omega}_{p+1}=0$ will be satisfied only if $d_c^2\overline{\alpha}=0$, since $d_c^2\overline{\alpha}-d_c^1\overline{\omega}_{p+1}=0$.
\subsubsection*{r=4} In this case, we have $\alpha\in Z_4^{p,\bullet}$ if there exist $\beta_p\in\C_{p,\bullet}$ and $z_{p+i}\in\C_{p+i,\bullet}$ for $i=1,2,3$ such that
\begin{align*}
    \alpha=d_0\beta_p+\overline{\alpha}\,,\, d_1\alpha+d_0z_{p+1}=d_2\alpha+d_1z_{p+1}+d_2z_{p+2}=d_3\alpha+d_2z_{p+1}+d_1z_{p+2}+d_0z_{p+3}=0\,.
\end{align*}
If we impose $\omega_{p+1}=z_{p+1}+d_1\beta_p$,  $\omega_{p+2}=z_{p+2}+d_2\beta_p-d_1\beta_{p+1}$, and $\omega_{p+3}=z_{p+3}+d_3\beta_p-d_2\beta_{p+1}-d_1\beta_{p+2}$, we have
\begin{itemize}
    \item $d_c^1\overline{\alpha}=d_c^2\overline{\alpha}-d_c^1\overline{\omega}_{p+1}=d_c^3\overline{\alpha}-d_c^2\overline{\omega}_{p+1}-d_c^1\overline{\omega}_{p+2}=0$;
    \item $\check{\omega}_{p+1}=d_0\beta_{p+1}$, $\check{\omega}_{p+2}=d_0\beta_{p+2}$, and $\check{\omega}_{p+3}=d_0\beta_{p+3}$ for some $\beta_{p+i}\in\C_{p+i,\bullet}$ with $i=1,2,3$;
    \item $\hat{\omega}_{p+1}=d_0^{-1}d_1\overline{\alpha}$, $\hat{\omega}_{p+2}=d_0^{-1}\left(\partial_2\overline{\alpha}-d_1\overline{\omega}_{p+1}\right)$, and $\hat{\omega}_{p+3}=d_0^{-1}\left(\partial_3\overline{\alpha}-\partial_2\overline{\omega}_{p+1}-d_1\overline{\omega}_{p+2}\right)$
\end{itemize}
and we get
\begin{align*}
    d_4\alpha-d_3z_{p+1}-d_2z_{p+2}-&d_1z_{p+3}=d_4(d_0\beta_p+\overline{\alpha})-d_3\left(d_0\beta_{p+1}+\overline{\omega}_{p+1}+d_0^{-1}d_1\overline{\alpha}-d_1\beta_p\right)+\\&-d_2\left[d_0\beta_{p+2}+\overline{\omega}_{p+2}+d_0^{-1}\left(\partial_2\overline{\alpha}-d_1\overline{\omega}_{p+1}\right)-d_2\beta_p+d_1\beta_{p+1}\right]+\\&-d_1\left[d_0\beta_{p+3}+\overline{\omega}_{p+3}+d_0^{-1}\left(\partial_3\overline{\alpha}-\partial_2\overline{\omega}_{p+1}-\partial_1\overline{\omega}_{p+2}\right)-d_3\beta_{p}+d_2\beta_{p+1}+d_1\beta_{p+2}\right]\\=&\partial_4\overline{\alpha}-\partial_3\overline{\omega}_{p+1}-\partial_2\overline{\omega}_{p+2}-d_1\overline{\omega}_{p+3}-d_0d_4\beta_p+d_0d_3\beta_{p+1}+d_0d_2\beta_{p+2}+d_0d_1\beta_{p+3}\\=&d_c^4\overline{\alpha}-d_c^3\overline{\omega}_{p+1}-d_c^2\overline{\omega}_{p+2}-d_c^1\overline{\omega}_{p+3}+d_0d_0^{-1}\left(\partial_4\overline{\alpha}-\partial_3\overline{\omega}_{p+1}-\partial_2\overline{\omega}_{p+2}-d_1\overline{\omega}_{p+3}\right)+\\&-d_0d_4\beta_p+d_0d_3\beta_{p+1}+d_0d_2\beta_{p+2}+d_0d_1\beta_{p+3}\,.
\end{align*}
On the other hand, an element $x\in B_4^{p+4,\bullet}$ if there exist $c_{p+1}\in Z_{3}^{p+1,\bullet}$ with $c_{p+1}=d_0\beta_{p+1}+\overline{c}_{p+1}$ for some $\beta_{p+1}\in\C_{p+1,\bullet}$, as well as $\overline{\omega}_{p+2}\in\ker\Box_0\cap\C_{p+2,\bullet}$ such that $d_c^1\overline{c}_{p+1}=d_c^2\overline{c}_{p+1}-d_c^1\overline{\omega}_{p+2}=0$, and
\begin{align*}
    x=d_c^3\overline{c}_{p+1}+d_0d_0^{-1}\partial_3\overline{c}_{p+1}-d_0d_3\beta_{p+1}-\sum_{i=1}^2\left(d_c^i\overline{\omega}_{p+4-i}+d_0d_0^{-1}\partial_i\overline{\omega}_{p+4-i}+d_0d_i\beta_{p+4-i}\right)+d_0c_{p+4}
\end{align*}
for some $\omega_{p+3}\in\C_{p+3,\bullet}$ and $\check{\omega}_{p+i}=d_0\beta_{p+i}$ for $i=2,3$. Therefore, if we choose
\begin{itemize}
    \item $c_{p+1}=-d_0\beta_{p+1}$, i.e. $\overline{c}_{p+1}=0$,
    \item $c_{p+2}=-d_1\beta_{p+1}+d_0\beta_{p+2}$, i.e. $\overline{\omega}_{p+2}=0$,
    \item $c_{p+3}=-d_2\beta_{p+1}+d_1\beta_{p+2}+d_0\beta_{p+3}+\overline{\omega}_{p+3}$, and
    \item $c_{p+4}=d_0^{-1}\left(\partial_4\overline{\alpha}-\partial_3\overline{\omega}_{p+1}-\partial_2\overline{\omega}_{p+2}\right)-d_4\beta_p$
\end{itemize}
we have
\begin{align*}
    \Delta_4\left(\left[\overline{\alpha}\right]\right)=&[d_c^4\overline{\alpha}-d_c^3\overline{\omega}_{p+1}-d_c^2\overline{\omega}_{p+2}-d_c^1\overline{\omega}_{p+3}+d_0d_0^{-1}\left(\partial_4\overline{\alpha}-\partial_3\overline{\omega}_{p+1}-\partial_2\overline{\omega}_{p+2}-\partial_1\overline{\omega}_{p+3}\right)+\\&-d_0d_4\beta_p+d_0\left(d_3\beta_{p+1}+d_2\beta_{p+2}+d_1\beta_{p+3}\right)]=\left[d_c^4\overline{\alpha}-d_c^3\overline{\omega}_{p+1}-d_c^2\overline{\omega}_{p+2}\right]\,.
\end{align*}
Again, just like in the previous case, unless we have $d_c^1\overline{\omega}_{p+1}=d_c^2\overline{\omega}_{p+1}+d_c^1\overline{\omega}_{p+2}=0$, the expression of the differential $\Delta_4([\overline{\alpha}])$ does not simplify to $[d_c^4\overline{\alpha}]$. Notice that this would also imply that $d_c^2\overline{\alpha}=d_c^3\overline{\alpha}=0$.

\begin{proposition}\label{Delta_r with Rumin}
    Let us prove that for any $r\ge 1$, given an arbitrary $\alpha\in Z_{r}^{p,\bullet}$ the action of the $r^{th}$ page of the spectral sequence is given by
    \begin{align}
        \Delta_r\left([\alpha]\right)=\left[d_c^r\overline{\alpha}-\sum_{i=2}^{r-1}d_c^{i}\overline{\omega}_{p+r-i}\right]\,,
    \end{align}
    where $\overline{\omega}_{p+r-i}\in\ker\Box_0\cap\C_{p+r-i,\bullet}$ such that $d_c^j\overline{\alpha}-\sum_{i=1}^{j-1}d_c^i\overline{\omega}_{p+j-i}=0$ for $j=1,\ldots,r-1$.
\end{proposition}
\begin{proof}
    Let $\alpha\in Z_r^{p,\bullet}$, then following that was proven in Proposition \ref{prop expression Z_r} we have that $\alpha=d_0\beta_p+\overline{\alpha}$ for some $\beta_p\in\C_{p,\bullet}$. Moreover, for each $i=1,\ldots,r-1$, there exist $z_{p+i},\,\beta_{p+i}\in\C_{p+i,\bullet}$ such that
    \begin{itemize}
        \item $\omega_{p+i}=z_{p+i}+d_i\beta_p-\sum_{j=1}^{i-1}d_{i-j}\beta_{p+j}$,
        \item $\check{\omega}_{p+i}=d_0\beta_{p+i}$ and $\hat{\omega}_{p+i}=d_0^{-1}\left(\partial_i\overline{\alpha}-\sum_{j=1}^{i-1}\partial_{i-j}\overline{\omega}_{p+j}\right)$.
    \end{itemize}
    Hence
    \begin{align*}
        d_r\alpha-\sum_{i=1}^{r-1}d_{r-i}z_{p+i}=&d_rd_0\beta_p+d_r\overline{\alpha}-\sum_{i=1}^{r-1}d_{r-i}\omega_{p+i}+\sum_{i=1}^{r-1}d_{r-i}d_i\beta_p-\sum_{i=1}^{r-1}d_{r-i}\sum_{j=1}^{i-1}d_{i-j}\beta_{p+j}\\=&d_rd_0\beta_p+d_r\overline{\alpha}-\sum_{i=1}^{r-1}d_{r-i}d_0\beta_{p+i}-\sum_{i=1}^{r-1}d_{r-i}\overline{\omega}_{p+i}-\sum_{i=1}^{r-1}d_{r-i}d_0^{-1}\partial_i\overline{\alpha}+\\&+\sum_{i=1}^{r-1}d_{r-i}d_0^{-1}\sum_{j=1}^{i-1}\partial_{i-j}\overline{\omega}_{p+j}+\sum_{i=1}^{r-1}d_{r-i}d_i\beta_p-\sum_{i=1}^{r-1}d_{r-i}\sum_{j=1}^{i-1}d_{i-j}\beta_{p+j}\\=&\partial_r\overline{\alpha}-\sum_{i=1}^{r-1}\partial_{r-i}\overline{\omega}_{p+i}+d_0\sum_{i=1}^{r-1}d_{r-i}\beta_{p+i}-d_0d_r\beta_p\\=&d_c^r\overline{\alpha}+d_0d_0^{-1}\partial_r\overline{\alpha}+d_c^{r-1}(-\overline{\omega}_{p+1})+d_0d_0^{-1}\partial_{r-1}(-\overline{\omega}_{p+1})-d_0d_{r-1}(-\beta_{p+1})+\\&-\sum_{i=2}^{r-1}d_c^{r-i}\overline{\omega}_{p+i}-d_0d_0^{-1}\sum_{i=2}^{r-1}\partial_{r-i}\overline{\omega}_{p+i}+d_0\sum_{i=2}^{r-1}d_{r-i}\beta_{p+i}-d_0d_r\beta_p\,.
    \end{align*}
    Since by Proposition \ref{prop expression B_r}, an element $x\in\C_{p+r,\bullet}$ belongs to $B_r^{p+r,\bullet}$ if
    \begin{align*}
    x=&d_c^{r-1}\overline{c}_{p+1}-\sum_{i=1}^{r-2}d_c^{r-1-i}\overline{\omega}_{p+1+i}+d_0d_0^{-1}\left(\partial_{r-1}\overline{c}_{p+1}-\sum_{i=1}^{r-2}\partial_{r-1-i}\overline{\omega}_{p+1+i}\right)+\\&-d_0d_{r-1}\beta_{p+1}+d_0\sum_{i=1}^{r-2}d_i\beta_{p+r-i}+d_0c_{p+r}
\end{align*}
where
\begin{itemize}
    \item $c_{p+1}=d_0\beta_{p+1}+\overline{c}_{p+1}\in\ker d_0$ for some $\beta_{p+1}\in\C_{p+1,\bullet}$;
    \item for each $i=1,\ldots,r-2$, we take $\omega_{p+i}\in\C_{p+i,\bullet}$ such that $\check{\omega}_{p+i}=d_0\beta_{p+i}$ for some $\beta_{p+i}\in\C_{p+i,\bullet}$.
\end{itemize}
It is enough to take
\begin{itemize}
    \item $c_{p+1}=-d_0\beta_{p+1}$,
    \item $c_{p+i}=-d_{i-1}\beta_{p+1}+\sum_{j=0}^{i-2}d_j\beta_{p+i-j}$ for $i=2,\ldots,r-2$,
    \item $c_{p+r-1}=-d_{r-2}\beta_{p+1}+\sum_{j=0}^{r-3}d_j\beta_{p+r-1-j}+\overline{\omega}_{p+r-1}$,
    \item $c_{p+r}=d_0^{-1}\partial_r\overline{\alpha}-d_r\beta_p$,
\end{itemize}
to have the claim
\begin{align*}
    \Delta_r\left([\alpha]\right)=&\left[d_r\alpha-\sum_{i=1}^{r-1}d_{r-i}z_{p+i}\right]=\left[d_c^r\overline{\alpha}-\sum_{i=2}^{r-1}d_c^{i}\overline{\omega}_{p+r-i}\right]\,
\end{align*}
where $\overline{\omega}_{p+r-i}\in\ker\Box_0\cap\C_{p+r-i}$ such that $d_c^j\overline{\alpha}-\sum_{i=1}^{j-1}d_c^i\overline{\omega}_{p+j-i}=0$ for $j=1,\ldots,r-1$.
\end{proof}
\section{Complexes from spectral sequences}\label{section 4}
The aim of this section is to extract subcomplexes from our truncated $s$-multicomplex by relying on the properties of the spectral sequences that we covered in the previous section.

Our approach is very different to the one behind the Rumin complex, where only one subcomplex is extracted from the original total complex $(\mathrm{Tot}\,\C,d=d_0+d_1+\cdots+d_s)$. The goal here is instead to extract ``as many complexes as possible''. This approach stems from the particular issues arising in the subRiemannian setting where the de Rham complex is an $s$-multicomplex.

As already shown, there is an isomorphism between the space of Rumin forms $E_0^\bullet$ and the quotients $E_1^{\bullet,\bullet}$ appearing at the first page of the spectral sequence. Using the scalar product introduced in Subsection \ref{scalar product}, we can say
\begin{align*}
    E_0^k=\bigoplus_pZ_1^{p,k-p}\cap\left(B_1^{p,k-p}\right)^\perp\,.
\end{align*}
In general, given an arbitrary choice of bidegree $(p,k)$, the Rumin differential on $E_0^k\cap\mathcal{C}_{p,k-p}$ will be given by a sum of several operators. Namely, there exists $I_{p,k}=\lbrace\nu_1,\ldots,\nu_N\rbrace$ with each $\nu_i\in\mathbb N$ and $\nu_1<\cdots<\nu_N$ such that
\begin{align*}
    d_c=\sum_{j\in I_{p,k}}d_c^j\colon E_0^{k}\cap\C_{p,k-p}\longrightarrow\bigoplus_{j\in I_{p,k}}E_0^{k+1}\cap\C_{p+j,k+1-p-j}
\end{align*}
where each $d_c^j$ has bidegree $(j,1-j)$.
Equivalently, this can be expressed by saying that at each page $j\in I_{p,k}$ the differentials $\Delta_j$ are non-trivial maps (see Proposition \ref{Delta_r with Rumin})
\begin{align*}
    \Delta_j\colon Z_j^{p,k-p}/B_j^{p,k-p}\longrightarrow Z_j^{p+j,k+1-p-j}/B_j^{p+j,k+1-p-j}\,.
\end{align*}
When working with Carnot groups, having to deal with such a situation becomes rather problematic: not only is the operator $d_c=\sum_{j\in I_{p,k}}d_c^j$ not invariant under the choice of the scalar product, but the complex $(E_0^\bullet,d_c)$ also does not satisfy the Rockland condition \cite{helffer1980hypoellipticite}. This is the main reason why we are choosing to take as many complexes as possible. 

\begin{lemma}\label{lemma inclusions}
    Given the graded submodules presented in Definition \ref{Z and B defined}, we have the following set of inclusions
    \begin{align*}
        B_1^{p,\bullet}\subseteq B_2^{p,\bullet}\subseteq\cdots\subseteq B_\infty^{p,\bullet}\subseteq Z_\infty^{p,\bullet}\subseteq \cdots\subseteq Z_2^{p,\bullet}\subseteq Z_1^{p,\bullet}\ \text{ for any }p\in\mathbb N\,.
    \end{align*}
\end{lemma}
\begin{proof}
    While the two sets of inclusions
    \begin{align*}
        B_1^{p,\bullet}\subseteq B_2^{p,\bullet}\subseteq\cdots \subseteq B_\infty^{p,\bullet}\ \text{ and }Z_\infty^{p,\bullet}\subseteq\cdots\subseteq Z_2^{p,\bullet}\subseteq Z_1^{p,\bullet}
    \end{align*}
    follow directly from the definition of these submodules, we focus on explicitly proving the inclusions
    \begin{align*}
        B_j^{p,\bullet}\subseteq Z_l^{p,\bullet}\ \text{ for any }j,\,l\,\in\mathbb N \text{ and any }p\in\mathbb N\,.
    \end{align*}
    These computations will come in handy to prove Proposition \ref{main prop}.

    $j=1$. Let $\alpha\in B_1^{p,\bullet}$, that is $\alpha=d_0c_p$ for some $c_p\in\C_{p,\bullet}$. We want to prove that $\alpha\in Z_l^{p,\bullet}$ for $l\ge 1$. The case of $l=1$ is trivial, since $\alpha\in\mathrm{Im}\,d_0\subset \ker d_0$, hence $\alpha\in Z_1^{p,\bullet}$. For $l\ge 2$, it is sufficient to use formula \eqref{d0dr} to get
    \begin{align*}
        d_m\alpha=&d_md_0c_p=-\sum_{i=1}^md_{m-i}d_ic_p=\sum_{i=1}^md_{m-i}\left(-d_{i}c_p\right)\ \text{ for any }m\in\mathbb N\,,
    \end{align*}
    hence $\alpha\in Z_l^{p,\bullet}$ with $z_{p+i}=-d_ic_p$ for any $i=1,\ldots,l$.

    $j=2$. Let $\alpha\in B_2^{p,\bullet}$, that is $d_0c_{p-1}=0$ and $d_1c_{p-1}+d_0c_{p}=\alpha$ for some $c_{p-i}\in\C_{p-i,\bullet}$ with $i=0,1$. We want to prove that $\alpha\in Z_l^{p,\bullet}$ for $l\ge 1$. If $l=1$, then
    \begin{align*}
        d_0\alpha=d_0(d_1c_{p-1}+d_0c_p)=-d_1d_0c_{p-1}+d_0^2c_p=0 \text{ since }d_0c_{p-1}=0\text{ and }d_0^2=0\,.
    \end{align*}
    For $l\ge 3$, we use an analogous formula as before, since
    \begin{align*}
        d_m\alpha=&d_md_1c_{p-1}+d_md_0c_p=-d_{m+1}d_0c_{p-1}-\sum_{i=1}^md_{m-i}d_{i+1}c_{p-1}-\sum_{i=1}^md_{m-i}d_ic_p\\=&-\sum_{i=1}^{m}d_{m-i}\left(d_{i+1}c_{p-1}+d_ic_p\right)\ \text{ for any }m\in\mathbb N\,,
    \end{align*}
    and so $\alpha\in Z_l^{p,\bullet}$ with $z_{p+i}=-d_{i+1}c_{p-1}-d_ic_p$ for $i=1,\ldots,l$.

    $j\ge 2$. Let $\alpha\in B_j^{p,\bullet}$, then there exist $c_{p-i}\in\C_{p-i,\bullet}$ for $i=0,\ldots,j-1$ such that
    \begin{align*}
        \sum_{k=i}^{j-1}d_{k-i}c_{p-k}=0\ \text{ for }1\le i\le j-1\text{ and }\alpha=\sum_{k=0}^{j-1}d_kc_{p-k}\,.
    \end{align*}
    For each $m\in\mathbb N$, we have that
    \begin{align*}
        d_md_k=-\sum_{i=1}^{k}d_{m+i}d_{k-i}-\sum_{i=1}^{m} d_{m-i}d_{k+i}
    \end{align*}
    so that
    \begin{align*}
        d_m\alpha=&\sum_{k=0}^{j-1}d_md_kc_{p-k}=d_md_0c_p-\sum_{k=1}^{j-1}\left(\sum_{i=1}^{k}d_{m+i}d_{k-i}+\sum_{i=1}^md_{m-i}d_{k+i}\right)c_{p-k}\\=&-\sum_{i=1}^{m}d_{m-i}d_ic_p-\sum_{i=1}^{j-1}\sum_{k=i}^{j-1}d_{m+i}d_{k-i}c_{p-k}-\sum_{i=1}^m\sum_{k=1}^{j-1}d_{m-i}d_{k+i}c_{p-k}\\=&-\sum_{i=1}^md_{m-i}\left(d_ic_p+\sum_{k=1}^{j-1}d_{k+i}c_{p-k}\right)-\sum_{i=1}^{j-1}d_{m+i}\underbrace{\left(\sum_{k=i}^{j-1}d_{k-i}c_{p-k}\right)}_{=0}=-\sum_{i=1}^md_{m-i}\sum_{k=0}^{j-1}d_{k+i}c_{p-k}
    \end{align*}
    hence $\alpha\in Z_l^{p,\bullet}$ with $z_{p+i}=-\sum_{k=0}^{j-1}d_{k+i}c_{p-k}$ for $i=1,\ldots,l$.
\end{proof}

\begin{proposition}\label{main prop}

    Let us consider a fixed bidegree $(p,k)$ and assume $I_{p,k}=\lbrace\nu_1,\ldots,\nu_N\rbrace$, that is
    \begin{align*}
        d_c^j\colon E_0^k\cap\C_{p,k-p}\longrightarrow E_0^{k+1}\cap\C_{p+j,k+1-p-j}
    \end{align*}
    is a non trivial map of bidegree $(j,1-j)$ for each $j\in I_{p,k}$.
    
    Then for each $j\in I_{p,k}$, the map
    \begin{align}\label{1st line}
       \begin{split}
           \Delta_j\colon &Z_j^{p,k-p}/B_{m_1}^{p,k-p}\longrightarrow Z_{m_2}^{p+j,k+1-p-j}/B_j^{p+j,k+1-p-j}\\
        &\Delta_j\left([\alpha]\right)=\left[d_j\alpha-\sum_{i=1}^{j-1}d_{j-i}z_{p+i}\right]=\left[d_c^j\overline{\alpha}-\sum_{i=2}^{j-1}d_c^{i}\overline{\omega}_{p+j-i}\right]
       \end{split} 
    \end{align}
    is well-defined for any choice of $m_1,m_2\in\mathbb N$. 

    Moreover, if for a fixed $(p,k)$ we have that both
\begin{align*}
    d_c^{j}\colon E_0^{k}\cap\C_{p,k-p}\longrightarrow E_0^{k+1}\cap\C_{p+{j},k+1-p-{j}}\ \text{ and }\  d_c^{l}\colon E_0^{k+1}\cap\C_{p+{j},k+1-p-{j}}\longrightarrow E_0^{k+2}\cap\C_{p+{j}+l,k+2-p-{j}-l}
\end{align*}
are non trivial maps, then for any $m_1,m_2\in\mathbb N$

\begin{align*}
    Z_j^{p,k-p}/B_{m_1}^{p,k-p}\xrightarrow[]{\Delta_j}Z_l^{p+j,k+1-p-j}/B_j^{p+j,k+1-p-j}\xrightarrow[]{\Delta_l}Z_{m_2}^{p+j+l,k+2-p-j-l}/B_l^{p+j+l,k+2-p-j-l}
\end{align*}
we have $\Delta_l\circ\Delta_j=0$.
\end{proposition}
\begin{proof}
    Using Lemma \ref{lemma inclusions}, we know that given an arbitrary $m_1\in\mathbb N$ we have the inclusion $B_{m_1}^{p,k-p}\subseteq Z_{l}^{p,k-p}$ for any $l\ge 1$. In particular, $$\alpha\in B_{m_1}^{p,k-p}\subseteq Z_{j+1}^{p,k-p}\ \text{ with }\ z_{p+i}=-\sum_{k=0}^{m_1-1}d_{k+i}c_{p-k} \text{ for }i=1,\ldots,j+1$$
    so that
    \begin{align*}
        d_j\alpha-\sum_{i=1}^{j-1}d_{j-i}z_{p+i}=d_j\sum_{k=0}^{m_1-1}d_kc_{p-k}+\sum_{i=1}^{j-1}\sum_{k=0}^{m_1-1}d_{j-i}d_{k+i}c_{p-k}=-d_0\sum_{k=0}^{m_1-1}d_{k+j}c_{p-k}\in B_1^{p+j,k+1-p-j}\,.
    \end{align*}
    Well-definedness with respect to the choice of witnesses $z_{p+i}$ follows from Theorem 2.10 in \cite{livernet2020spectral}, since the class $[d_j\alpha-\sum_{i=1}^{j-1}d_{j-i}z_{p+i}]\in E_j^{p+j,\bullet}$ is independent of the chosen $z_{p+i}$.

    Similarly, since $$d_j\alpha-\sum_{i=1}^{j-1}d_{j-i}z_{p+i}\in B_{j+1}^{p+j,k+1-p-j}\subseteq Z_{l+1}^{p+j,k+1-p-j}\ \text{ for any }\ \alpha\in Z_j^{p,k-p}\,,$$we can repeat the same reasoning to get $\Delta_l\circ\Delta_{j}=0$.

    Finally, equation \eqref{1st line} follows directly from Proposition \ref{Delta_r with Rumin}.
\end{proof}

    Notice how in Proposition \ref{main prop} each non-trivial map $d_c^j$ (or equivalently $\Delta_j$) of bidegree $(j,1-j)$ determines the page of the $Z$ submodule in the domain and the $B$ submodule in the target, i.e.
    \begin{align*}
        \Delta_j\colon Z_j^{p,\bullet}/B_{m_1}^{p,\bullet}\longrightarrow Z_{m_2}^{p+j,k+1-p-j}/B_j^{p+j,k+1-p-j}\,.
    \end{align*}
    Here $m_1,m_2\in\mathbb N$ are taken to be arbitrary, however if we are interested in constructing a complex then $m_1$ will be determined by the bidegree $(m_1,1-m_1)$ of the incoming map $\Delta_{m_1}$, while $m_2$ will instead be determined by the bidegree $(m_2,1-m_2)$ of the outgoing map $\Delta_{m_2}$, i.e.
    \begin{align*}
        \xrightarrow[]{\Delta_{k_1}}Z_{m_1}^{p-m_1,k-1-p+m_1}/&B_{k_1}^{p-m_1,k-1-p+m_1}\xrightarrow[]{\Delta_{m_1}}Z_j^{p,k-p}/B_{m_1}^{p,k-p}\xrightarrow[]{\Delta_j}Z_{m_2}^{p+j,k+1-p-j}/B_j^{p+j,k+1-p-j}\longrightarrow\\ &\xrightarrow[]{\Delta_{m_2}}Z_{k_2}^{p+j+m_2,k+2-p-j-m_2}/B_{m_2}^{p+j+m_2,k+2-p-j-m_2}\xrightarrow[]{\Delta_{k_2}}
    \end{align*}
    The reasoning can be extended to include each degree $k\in\mathbb N$.

    In this way, we are able to construct as many subcomplexes as the elements $j\in I_{k,p}$ for each $(p,k)$ where $E_1^{p,k-p}$ is non trivial.

    Since we are interested in defining these complexes as subspaces of forms and not as quotients like in \eqref{1st line}, we will rely on the scalar product introduced in Subsection \ref{scalar product} to define them.
    \begin{definition}\label{subc as subsp}[Spectral complexes associated with $\C$] Given an $s$-multicomplex, if for some bidegree $(p,k)$ we have $E_1^{p,k-p}$ is non trivial (or equivalently $\ker\Box_0\cap\C_{p,k-p}\neq 0$), then for each $j\in I_{p,k}$ and $l\in I_{p-l,k-1}$, we can define
        \begin{align}\label{notation spaces}
        E_{j,l}^{p,k-p}:=Z_{j}^{p,k-p}\cap\left(B_l^{p,k-p}\right)^\perp\subset \ker\Box_0\cap\C_{p,k-p}\,.
    \end{align}
    Moreover, for each $j\in I_{p,k}$, $l\in I_{p-l,k-1}$ and $i\in I_{p+j,k+1}$
    \begin{align*}
        E_{l,m_1}^{p-l,k-1-p+l}\xrightarrow[]{\Delta_l}E_{j,l}^{p,k-p}\xrightarrow[]{\Delta_j}E_{i,j}^{p+j,k+1-p-j}\xrightarrow[]{\Delta_i}E_{m_2,i}^{p+j+i,k+2-p-j-i}
    \end{align*}
    satisfies $\Delta_j\circ \Delta_l=\Delta_i\circ \Delta_j=0$ for any (non trivial) choice of $m_1,\,m_2\in\mathbb N$.

    To make the definition even more explicit, when we write
    \begin{align*}
        \Delta_j\colon E_{j,l}^{p,k-p}\xrightarrow[]{} E_{m,j}^{p+j,k+1-p-j}
    \end{align*}
    we mean the action of the operator \eqref{1st line}, projected on the orthogonal complement of $B_j^{p+j,k+1-p-j}$.

    Since this construction applies in every degree $k$, it yields a family of subspaces $E_{j,l}^{\bullet,\bullet}$. These subspaces are connected by the operators $\Delta_j$ of bidegree $(j,1-j)$, hence forming a collection of complexes.

    As the definition $E_{j,l}^{\bullet,\bullet}$ is given using spectral sequences, we  will refer to the collection of such subcomplexes as the \textit{spectral complexes} associated with the $s$-multicomplex $\C$ and denote it by
    \begin{align*}
        \bigg\lbrace \left( E_{j,l}^{\bullet,\bullet},\Delta_j\right)\bigg\rbrace_{j\in I_{\bullet,\bullet}}\,.
    \end{align*}
    \end{definition}
    \subsection{Considerations on spectral complexes}\label{considerations}
    The construction presented in Definition \ref{subc as subsp} consists of a family of complexes. Notably, for each $(p,k)$ such that $E_0^k\cap\C_{p,k-p}$ is non trivial (or equivalently $E_1^{p,k-p}$ is non trivial), there will be \textit{at least} one subcomplex that contains a subspace of Rumin forms of bidegree $(p,k)$.

    Indeed, given the series of inclusions shown in Lemma \ref{lemma inclusions}, we have that for any $j,l\ge 1$
    \begin{align*}
        Z_j^{p,k-p}\cap\left(B_l^{p,k-p}\right)^\perp\subseteq Z_1^{p,k-p}\cap\left(B_1^{p,k-p}\right)^\perp=\ker\Box_0\cap\C_{p,k-p}\,.
    \end{align*}
    In other words, for each non trivial choice of $(p,k)$ and any $j,l\ge 1$, we have that $E_{j,l}^{p,k-p}\subseteq E_0^k\cap\C_{p,k-p}$, i.e. it is a subspace of Rumin forms of degree $k$ and weight $p$.

    More interestingly, for any two non trivial choices of bidegrees $(p_1,k)$ and $(p_2,k+1)$ with $p_1<p_2$, there exists at least one complex that, when restricted to degree $k$, will look like
    \begin{align*}
        \Delta_{p_2-p_1}\colon E_{p_2-p_1,m_1}^{p_1,k-p_1}=Z_{p_2-p_1}^{p_1,k-p_1}\cap\left(B_{m_1}^{p_1,k-p_1}\right)^\perp\longrightarrow E_{m_2,p_2-p_1}^{p_2,k+1-p_2}=Z_{m_2}^{p_2,k+1-p_2}\cap\left(B_{p_2-p_1}^{p_2,k+1-p_2}\right)^\perp
    \end{align*}
    for some appropriate choice of $m_1,m_2\in\mathbb N$.

    \begin{remark}\label{charact spectral subs in terms of Rumin} It is possible to streamline the characterisation of the subspaces $E_{j,l}^{p,k-p}$ using the formulae found in Propositions \ref{prop expression Z_r} and \ref{prop expression B_r} in terms of the differentials $d_c^j$ (here we are assuming $j,l\ge 2$):
    \begin{align*}
        \overline{\alpha}\in Z_j^{p,k-p}\ \Longleftrightarrow\ & \exists\,\overline{\omega}_{p+i}\in\ker\Box_0\cap\C_{p+i,k-p-i}\text{ with }i=1,\ldots,j-2\text{ such that}\\
        &\sum_{i=1}^{j-1}\left(d_c^i\overline{\alpha}-\sum_{h=1}^{i-1}d_c^h\overline{\omega}_{p+i-h}\right)=0\\
        \overline{\beta}\in B_l^{p,k-p}\ \Longleftrightarrow\ &\exists\,\overline{c}_{p-i}\in\ker\Box_0\cap\C_{p-i,k-1-p+i}\text{ with }i=1,\ldots,l-1\text{ such that}\\ &\sum_{i=1}^{l-1}\left(d_c^i\overline{c}_{p-l+1}-\sum_{h=1}^{i-1}d_c^h\overline{c}_{p-l+1-h+i}\right)=\overline{\beta}
    \end{align*}
         If we introduce the notation
    \begin{align*}
        d_c^j\big\vert_p:=d_c^j\colon\ker\Box_0\cap\C_{p,\bullet}\longrightarrow\ker\Box_0\cap\C_{p+j,\bullet}
    \end{align*}
    to denote the restriction of the map $d_c^j$ to the Rumin forms of weight strictly equal to $p$, we get that
    \begin{align*}
        Z_{r}^{p,k-p}=&\ker d_c^1\big\vert_p \cap\ker\left(d_c^2\big\vert_p+d_c^1\big\vert_{p+1}\right)\cap\cdots\cap\ker\left(d_c^{r-1}\big\vert_p+d_c^{r-2}\big\vert_{p+1}+\cdots+d_c^1\big\vert_{p+r-2}\right)\\=&\bigcap_{j=1}^{r-1}\ker\left(d_c^j\big\vert_p+\sum_{i=1}^{j-1}d_c^{i}\vert_{p+j-i}\right)\\
        B_{r}^{p,k-p}=&\mathrm{Im}\,d_c^1\big\vert_{p-r+1}+\mathrm{Im}\,\left(d_c^2\big\vert_{p-r+1}+d_c^1\big\vert_{p-r+2}\right)+\cdots+\mathrm{Im}\,\left(d_c^{r-1}\big\vert_{p-r+1}+\cdots+d_c^1\big\vert_{p-1}\right)\\=&\sum_{j=1}^{r-1}\mathrm{Im}\,\left(d_c^j\big\vert_{p-r+1}+\sum_{i=1}^{j-1}d_c^i\big\vert_{p-r+1+j-i}\right)\,.
    \end{align*}
    \end{remark}

    In the case where for each degree $k=0,\ldots,n$, the space of Rumin forms is non trivial only in one weight, i.e. for each $k$, there exists only one $p=p(k)$ such that $E_1^{p,k-p}\neq 0$, then the spectral complexes construction extracts the Rumin complex.
    Indeed, assume that $E_1^{p_1,k-p_1}$ and $E_1^{p_2,k+1-p_2}$ are the only non trivial spaces in degree $k$ and $k+1$ respectively (here we are also implicitly assuming $p_2>p_1$). In this case, $\Delta_{p_2-p_1}$ will be the only non trivial operator and so we are interested in 
    \begin{align*}
        Z_{p_2-p_1}^{p_1,k-p_1}\cap \left(B_{m_1}^{p_1,k-p_1}\right)^\perp\xrightarrow[]{\Delta_{p_2-p_1}}Z_{m_2}^{p_2,k+1-p_2}\cap\left( B_{p_2-p_1}^{p_2,k+1-p_2}\right)^\perp\,.
    \end{align*}
    Then for any $p>p_1$, $\ker\Box_0\cap\C_{p,k-p}=0$, but also $\ker \Box_0\cap\C_{p,k+1-p}=0$ as long as $p<p_2$. Therefore
  \begin{align*}
      Z_{p_2-p_1}^{p_1,k-p_1}=&Z_1^{p_1,k-p_1}\cap\ker d_c^1\big\vert_{p_1}\cap \ker\left(d_c^2\big\vert_{p_1}+d_c^1\big\vert_{p_1+1}\right)\cap\cdots\\&\cap\ker\left(d_c^{p_2-p_1-1}\big\vert_{p_1}+d_c^{p_2-p_1-2}\big\vert_{p_1+1}+\cdots+d_c^1\big\vert_{p_1+p_2-p_1-2}\right)\\=&Z_1^{p_1,k-p_1}\cap\ker d_c^1\big\vert_{p_1}\cap\ker d_c^2\big\vert_{p_1}\cap\cdots\cap\ker d_c^{p_2-p_1-1}\big\vert_{p_1}=\ker d_0\cap\C_{p_1,k-p_1}\,,
  \end{align*}
  while 
  \begin{align*}
      B_{p_2-p_1}^{p_2,k+1-p_2}=&\mathrm{Im}\,d_0+\mathrm{Im}\,d_c^1\big\vert_{p_2-(p_2-p_1)+1}+\mathrm{Im}\,\left(d_c^2\big\vert_{p_2-(p_2-p_1)+1}+d_c^1\big\vert_{p_2-(p_2-p_1)+2}\right)+\cdots+\\&+\mathrm{Im}\,\left(d_c^{p_2-p_1-1}\big\vert_{p_2-(p_2-p_1)+1}+\cdots+d_c^1\big\vert_{p_2-1}\right)=\mathrm{Im}\,d_0=B_1^{p_2,k+1-p_2}\,.
  \end{align*}
  If we repeat the same reasoning to compute $B_{m_1}^{p_1,k-p_1}$ and $Z_{m_2}^{p_2,k+1-p_2}$, since we are assuming that also in degrees $k-1$ and $k+2$ the non trivial Rumin forms appear in one unique weight, we get that
  \begin{itemize}
      \item $Z_{p_2-p_1}^{p_1,k-p_1}\cap \left( B_{m_1}^{p_1,k-p_1}\right)^\perp=Z_{1}^{p_1,k-p_1}\cap \left( B_{1}^{p_1,k-p_1}\right)^\perp=\ker\Box_0\cap\C_{p_1,k-p_1}$,
      \item $Z_{m_2}^{p_2,k+1-p_2}\cap\left(B_{p_2-p_1}^{p_2,k+1-p_2}\right)^\perp=Z_{1}^{p_2,k+1-p_2}\cap\left(B_{1}^{p_2,k+1-p_2}\right)^\perp=\ker\Box_0\cap\C_{p_2,k+1-p_2}$, and
      \item $\Delta_{p_2-p_1}\colon\ker\Box_0\cap\C_{p_1,k-p_1}\longrightarrow\ker\Box_0\cap\C_{p_2,k+1-p_2}$ coincides with the Rumin differential $d_c^{p_2-p_1}$ of bidegree $(p_2-p_1,1-p_2+p_1)$.
  \end{itemize}
  By extending this type of reasoning, we also get that if at each degree $k=0,\ldots,n$ the spaces $E_1^{p,k-p}$ are non trivial for \textit{at most 2} possible weights, then the action of each non trivial $\Delta_r$ will coincide with the corresponding Rumin differential $d_c^r$ of bidegree $(r,1-r)$. In this more general case where different weights appear, however, the collection $\lbrace\left(E_{j,l}^{\bullet,\bullet},\Delta_j\right)\rbrace_{j\in I_{\bullet,\bullet}}$ will consist of multiple complexes and consequently the spaces $E_{j,l}^{p,k-p}$ will be strict subspaces of $\ker\Box_0\cap\C_{p,k-p}$ at least for some choice of $k$ and $p$.

  In the most general case, where for some degree $k$ the space $E_0^k$ is non trivial for three or more choices of weight, the operator $\Delta_r$ with $r\ge 3$ appears to be a sum of differentials $d_c^j$ of different orders. However, even in this case, each operator $\Delta_r\colon E_{r,m_1}^{p,\bullet}\longrightarrow E_{m_2,r}^{p+r,\bullet}$ is invariant under the choice of scalar product, since it coincides with the action of the differential operator arising from the action of $d$ on the $r^{th}$ page of the spectral sequence.

  Finally, let us consider the cohomology computed by such spectral complexes. 
In general, the direct sum of the cohomology of each subcomplex is not isomorphic to the cohomology of the total complex $(\mathrm{Tot}\,\C,d=d_0+d_1+\cdots+d_s)$. Indeed, given any $j,l\in\mathbb N$ such that the sequence
        \begin{align*}
            E_{l,m_1}^{p-l,k-1-p+l}\xrightarrow[]{\Delta_l}E_{j,l}^{p,k-p}\xrightarrow[]{\Delta_j}E_{m_2,j}^{p+j,k+1-p-j}
        \end{align*}
        is non trivial, then
        \begin{align*}
            \frac{\ker \Delta_j\colon E_{j,l}^{p,k-p}\longrightarrow E_{m_2,j}^{p+j,k+1-p-j}}{\mathrm{Im}\,\Delta_l\colon E_{l,m_1}^{p-l,k-1-p+l}\longrightarrow E_{j,l}^{p,k-p}}=\frac{Z_{j+1}^{p,k-p}}{B_{l+1}^{p,k-p}}\,.
        \end{align*}
        However, since we are working with a filtration by weight of finite depth (see Remark \ref{remark 1.2}), for each $(p,k)$ there will be a particular couple $(J,L)=(J_{p,k},L_{p,k})\in\mathbb N\times\mathbb N$ such that $Z_{J+1}^{p,k-p}=Z_\infty^{p,k-p}$ and $B_{L+1}^{p,k-p}=B_\infty^{p,k-p}$, and so for each degree $k$:
        \begin{align*}
            \bigoplus_p\frac{\ker d_c^J\colon E_{J,L}^{p,k-p}\longrightarrow E_{m_2,J}^{p+J,k+1-p-J}}{\mathrm{Im}\,d_c^{L}\colon E_{L,m_1}^{p-L,k-1-p+L}\longrightarrow E_{J,L}^{p,k-p}}=\bigoplus_p E_\infty^{p,k-p}=H^k(\mathrm{Tot}\,\C,d)\,.
        \end{align*}

\subsection{Hodge duality}
It is a well-known fact that the Rumin complex $(E_0^\bullet,d_c)$ is closed under Hodge-$\star$ duality. Our aim now is to show that, even though each one of the subcomplexes constructed above is \textit{not} Hodge-$\star$ closed, the whole collection $\lbrace\left(E_{j,l}^{\bullet,\bullet},d_c^j\right)\rbrace_{j\in I_{\bullet,\bullet}}$ is.

In order to introduce the Hodge-$\star$ operator to our commutative dg-algebra $(\mathrm{Tot}\,\C,\wedge,d=d_0+d_1+\cdots+d_s)$, we require a scalar product (as introduced in Subsection \ref{scalar product}), as well as a corresponding volume form. When dealing with $\mathrm{Tot}\,\C$ being the space of smooth differential forms on a (not necessarily compact) Riemannian manifold, it is useful to think about this in terms of the $L^2$-inner product on forms. We refer to \cite{F+T1} for a more thorough presentation in the case of simply connected nilpotent Lie groups.

\begin{definition}[Hodge-$\star$ operator]\label{Hodge star def}

Given a differential algebra $(\mathrm{Tot}\,\C,\wedge,d,\langle\cdot,\cdot\rangle,\mathrm{vol})$ with $(\mathrm{Tot}\,\C,d=d_0+d_1+\cdots+d_s)$ an $s$-multicomplex, $\langle\cdot,\cdot\rangle_k$ a scalar product on each $(\mathrm{Tot}\,C)_k$ such that elements of different weight are orthogonal, and $\mathrm{vol}\in\C_{Q,n-Q}$ the associated volume form, then we can define the Hodge-$\star$ operator as the linear operator such that
\begin{align*}
    \star\colon\C_{p,k-p}&\longrightarrow\C_{Q-p,n-k-Q+p}\\
    \alpha\wedge\star\beta:=\langle\alpha,&\beta\rangle_k\mathrm{vol}\ \text{ for any }\alpha,\beta\in\C_{p,k-p}\,.
\end{align*}
    
\end{definition}
From the definition, it easily follows that 
\begin{align}\label{properties of star}
    \alpha\wedge\star\beta=\beta\wedge\star\alpha\ ,\ \star\star\beta=(-1)^{k(n-k)}\beta\ \text{ and }\langle\star\alpha,\star\beta\rangle_{n-k}\ \text{ for any }\alpha,\beta\in(\mathrm{Tot}\,\C)_k\,.
\end{align}
Analogously to Definition \ref{adjoint of d0}, we can define the formal adjoint of the map $d$, as well as of each structural map $d_i$.
\begin{definition}[Adjoint of $d$]
    Using the scalar product $\langle\cdot,\cdot\rangle$, we can define the formal adjoint of $d$, which we will denote by $\delta$ as the map 
    \begin{align*}
        \delta&\colon\left(\mathrm{Tot}\,C\right)_k\longrightarrow\left(\mathrm{Tot}\,\C\right)_{k-1}\\ \langle d\alpha,\beta\rangle_k=\langle\alpha&,\delta\beta\rangle_{k-1}\ \text{ for any }\alpha\in(\mathrm{Tot}\,\C)_{k-1}\,,\beta\in(\mathrm{Tot}\,\C)_k
    \end{align*}
\end{definition}
\begin{remark}
     When dealing with compactly supported differential forms $\alpha$ and $\beta$ of degree $k-1$ and $k$ respectively on a Riemannian manifold $X$, Stokes' theorem together with the properties listed in \eqref{properties of star} also imply that
     \begin{align}\label{formula of delta}
        \delta=(-1)^{n(k+1)+1}\star d\star\colon\left(\mathrm{Tot}\,\C\right)_k\longrightarrow\left(\mathrm{Tot}\,\C\right)_{k-1}\,,
    \end{align}
    since
    \begin{align*}
        \int_Xd(\alpha\wedge\star\beta)=\int_Xd\alpha\wedge\star\beta+(-1)^{k-1}\int_X\alpha\wedge d(\star\beta)=\int_Xd\alpha\wedge\star\beta+(-1)^{k-1+(n-k+1)(k-1)}\int_{X}\alpha\wedge\star(\star d\star\beta)\,.
    \end{align*}
\end{remark}
\begin{lemma}\label{lemma 3.8}Given the $s$-multicomplex $(\mathrm{Tot}\,\C,\wedge,d,\langle\cdot,\cdot\rangle,\mathrm{vol})$ with $d=d_0+d_1+\cdots+d_s$, once we define
\begin{align*}
    \delta_i\colon&\left(\mathrm{Tot}\,\C\right)_{k}\longrightarrow\left(\mathrm{Tot}\,\C\right)_{k-1}\ ,\ i=0,1,\ldots,s\\
    \langle d_i\alpha,\beta\rangle_k=&\langle\alpha,\delta_i\beta\rangle_{k-1}\ \text{ for any }\alpha\in\left(\mathrm{Tot}\,\C\right)_{k-1},\,\beta\in\left(\mathrm{Tot}\,\C\right)_{k}\,,
\end{align*}
we have
\begin{align*}
    \delta_i\C_{p,k-p}\subset\C_{p-i,k-1-p+i}\ \text{ and }\delta_i=(-1)^{(k+1)n+1}\star d_i\star\,.
\end{align*}\end{lemma}
\begin{proof}
    For the first claim, it is enough to take $\alpha\in\C_{p,k-p}$ and any $\beta\in\left(\mathrm{Tot}\,\C\right)_{k-1}$ so that
    \begin{align*}
        \langle\delta_i\alpha,\beta\rangle_{k-1}=\langle\alpha,d_i\beta\rangle_k\neq 0 \ \Longleftrightarrow\ \beta\in\C_{p-i,k-1-p+i}\,. 
    \end{align*}
    For the second part, by linearity we have 
    \begin{align*}
        \sum_{i=0}^s\delta_i=\delta=(-1)^{n(k+1)+1}\star d\star=\sum_{i=0}^s(-1)^{n(k+1)+1}\star d_i\star
    \end{align*}
    and since elements of different weight are (orthogonal and hence) linearly independent, we have
    \begin{align*}
        \delta_i=(-1)^{n(k+1)+1}\star d_i\star\ \text{ for each }i=0,1,\ldots,s\,.
    \end{align*}
\end{proof}
    \begin{lemma}\label{Rumin forms closed under Hodge star}
        The space of Rumin forms $E_0^\bullet=\ker\Box_0\cap\mathrm{Tot}\,\C$ is closed under Hodge-$\star$, i.e.
        \begin{align*}
            \star\, E_0^{k}\cap\C_{p,k-p}=E_0^{n-k}\cap\C_{Q-p,n-k-Q+p}\,.
        \end{align*}
    \end{lemma}
    \begin{proof}
        The claim readily follows from the equalities
        \begin{align*}
            d_0\star\alpha=&(-1)^{(k+1)(n-k-1)}\star\star d_0\star\alpha=\pm \star\delta_0\alpha=0\\
            \delta_0\star\alpha=&(-1)^{n(n-k+1)+1}\star d_0\star\star\alpha=\pm \star d_0\alpha=0\,
        \end{align*}
        and the linearity of the Hodge-$\star$ map.
        
    \end{proof}

\begin{definition}[The adjoint of $d_c$] Using the scalar product and all the properties just seen of the formal adjoints of the maps $d_i$, we can define
\begin{align*}
    \delta_c\colon E_0^k&\longrightarrow E_0^{k-1}\\ \langle d_c\overline{\alpha},\overline{\beta}\rangle_{k}=\langle\overline{\alpha},\delta_c\overline{\beta}\rangle_{k-1}\ &\text{ for any }\alpha\in E_0^{k-1}\,,\,\beta\in E_0^{k}\,.
\end{align*}
    
\end{definition}
\begin{remark}Just like when dealing with the formal adjoint of the differential $d$, an analogous formula to \eqref{formula of delta} holds for the adjoint of the Rumin differential $d_c$, that is
\begin{align*}
    \delta_c=(-1)^{n(k+1)+1}\star d_c\star\colon E_0^k\longrightarrow E_0^{k-1}\,.
\end{align*}
The claim follows again by a non-trivial application of Stokes' theorem together with the properties listed in \eqref{properties of star}, since
\begin{align}
    \int_Xd_c\overline{\alpha}\wedge\overline{\beta}+(-1)^{k-1}\int_X\overline{\alpha}\wedge d_c\overline{\beta}=0 \ \text{ for any compactly supported }\overline{\alpha}\in E_0^{k-1},\,\overline{\beta}\in E_0^k\,.
\end{align}
We stress that, unlike the one used to show \eqref{formula of delta}, this ``integration by parts'' formula is not straightforward and we refer the reader to Remark 2.16 in \cite{CompensatedCompactness} for a proof of this statement.
    
\end{remark}
\begin{lemma}\label{adjoint little dc}
Given the $s$-multicomplex $(\mathrm{Tot}\,\C,\wedge,d,\langle\cdot,\cdot\rangle,\mathrm{vol})$ with $d=d_0+d_1+\cdots+d_s$, once we define
\begin{align*}
\delta_c^i\colon E_0^k\longrightarrow E_0^{k-1}\ ,\  i=1,\ldots,S\\
\langle d_c^i\overline{\alpha},\overline{\beta}\rangle_k=\langle\overline{\alpha},\delta_c^i\overline{\beta}\rangle_{k-1}\ \text{ for any }\overline{\alpha}\in E_0^{k-1},\,\overline{\beta}\in E_0^k\,,
\end{align*}
we have that
\begin{align*}
    \delta_c^i\left(\ker\Box_0\cap\C_{p,k-p}\right)\subset\ker\Box_{0}\cap\C_{p-i,k-1-p+i}\ \text{ and }\delta_c^i=(-1)^{n(k+1)+1}\star\delta_c^i\star\,.
\end{align*}
\end{lemma}
\begin{proof}
    The proof follows exactly the same steps as the proof of Lemma \ref{lemma 3.8}, so that given $\overline{\alpha}\in\ker\Box_0\cap\C_{p,k-p}$ and any $\overline{\beta}\in E_0^{k-1}$, we have
    \begin{align*}
        \langle\delta_c^i\overline{\alpha},\overline{\beta}\rangle_{k-1}=\langle\overline{\alpha},d_c^i\overline{\beta}\rangle_{k}\neq 0\ \Longleftrightarrow\ \overline{\beta}\in\ker\Box_0\cap\C_{p-i,k-1-p+i}\,.
    \end{align*}
    For the second part, by linearity and using the linear independence of elements of different weight, we have
    \begin{align*}
        \sum_{i=1}^S\delta_c^i=\delta_c=(-1)^{n(k+1)+1}\star d_c\star=\sum_{i=1}^S(-1)^{n(k+1)+1}\star d_c^i\star
    \end{align*}
    and so
    \begin{align*}
        \delta_c^i=(-1)^{n(k+1)+1}\star d_c^i\star\ \text{ for each }i=1,\ldots,S\,.
    \end{align*}
    Notice that, when dealing with the Rumin differential and consequently its adjoint, the range $i=1,\ldots,S$ starts from 1 (and not 0) and $S=N-1$ is independent of $s$ (as pointed out in Lemma \ref{expression of d_c}).
\end{proof}

 \begin{remark}
    Before venturing into whether the collection of subcomplexes $\lbrace\left(E_{j,l}^{\bullet,\bullet},\Delta_j\right)\rbrace_{j\in I_{\bullet,\bullet}}$ is closed under Hodge-$\star$ duality, let us first study whether the submodules $B_l^{p,\bullet}$ are closed subspaces of $\C_{p,\bullet}$.  This question comes from the fact that the spaces $Z_l^{p,\bullet}$, defined as the kernel of bounded linear operators, are closed subspaces of $\C_{p,\bullet}$, however the range $B_l^{p,\bullet}$ in general is not. 
    Throughout this paper, in order to simplify the computations, we will assume that each $B_l^{p,\bullet}$ is closed. This is a strong assumption, since even in the application case that we have in mind of the de Rham complex on Carnot groups this will not be the case in general, unless $l=1,\infty$.
\end{remark}

\begin{proposition}[Closure under Hodge-$\star$]\label{Hodge star closed} The spectral complexes $\lbrace\left(E_{j,l}^{\bullet,\bullet},\Delta_j\right)\rbrace_{j\in I_{\bullet,\bullet}}$ introduced in Definition \ref{subc as subsp} are Hodge-$\star$ closed, in particular for any $r_1,\,r_2\ge 1$ we have
\begin{align*}
    \star\left[Z_{r_1}^{p,k-p}\cap\left(B_{r_2}^{p,k-p}\right)^\perp\right]=Z_{r_2}^{Q-p,n-k-Q+p}\cap\left(B_{r_1}^{Q-p,n-k-Q+p}\right)^\perp\,.
\end{align*}
\end{proposition}
\begin{proof}
    The case when $r_1=r_2=1$ was already covered in Lemma \ref{Rumin forms closed under Hodge star}, so we can directly consider $r_1,r_2\ge 2$. 
    We can also simplify our considerations by taking the characterisation of the submodules $Z_{r_1}^{\bullet,\bullet}$ and $B_{r_2}^{\bullet,\bullet}$ in terms of Rumin forms and the differentials $d_c^j$ (see Remark \ref{charact spectral subs in terms of Rumin}).   Moreover, from the properties listed in Lemma \ref{adjoint little dc}, we have
    \begin{align*}
    \overline{\alpha}\in \ker d_c^i\big\vert_p \longleftrightarrow \star\overline{\alpha}\in\ker\delta_c^i\big\vert_{Q-p}\ ,\ \left(d_c^i\big\vert_p\right)^\ast=\delta_c^i\big\vert_{p+i}\ \text{ and }\left(\delta_c^i\big\vert_p\right)^\ast=d_c^i\big\vert_{p-i}\,,
    \end{align*}
    where the $\ast$ here denotes the formal adjoint (see also Definition \ref{adjoint little dc}). 
    Therefore, it is a straightforward to check
    \begin{align*}
        \star Z_{r}^{p,k-p}=&\bigcap_{j=1}^{r-1}\ker\left(\delta_c^j\big\vert_{Q-p}+\sum_{i=1}^{j-1}\delta_c^i\big\vert_{Q-p-j+i}\right)=\left[\sum_{j=1}^{r-1}\mathrm{Im}\,\left(d_c^j\big\vert_{Q-p-j}+\sum_{i=1}^{j-1}d_c^i\big\vert_{Q-p-j}\right)\right]^\perp\\
        =&\left[\sum_{j=1}^{r-1}\sum_{i=1}^j\mathrm{Im}\,d_c^j\big\vert_{Q-p-j}\right]^\perp=\left(\sum_{i=1}^{r-1}\sum_{k=Q-p-(r-1)}^{Q-p-i}\mathrm{Im}\,d_c^i\big\vert_{k}\right)^\perp\,.
    \end{align*}
    If we consider $B_{r}^{Q-p,n-k-Q+p}$, we get
    \begin{align*}
        B_r^{Q-p,n-k-Q+p}=&\sum_{j=1}^{r-1}\mathrm{Im}\,\left(d_c^j\big\vert_{Q-p-r+1}+\sum_{i=1}^{j-1}d_c^i\big\vert_{Q-p-r+1+j-i}\right)\\=&\sum_{j=1}^{r-1}\mathrm{Im}\,d_c^j\big\vert_{Q-p-r+1}+\sum_{i=1}^{r-1}\sum_{j=i+1}^{r-1}\mathrm{Im}\,d_c^i\big\vert_{Q-p-r+1+j-i}\\=&\sum_{i=1}^{r-1}\mathrm{Im}\,d_c^i\big\vert_{Q-p-r+1}+\sum_{i=1}^{r-1}\sum_{k=Q-p-r+2}^{Q-p-i}\mathrm{Im}\,d_c^i\big\vert_{k}=\sum_{i=1}^{r-1}\sum_{k=Q-p-r+1}^{Q-p-i}\mathrm{Im}\,d_c^i\big\vert_k
    \end{align*}
    which implies that
    \begin{align*}
        \star Z_r^{p,k-p}=\left(B_r^{Q-p,n-k-Q+p}\right)^\perp\,.
    \end{align*}
    By reasoning in exactly the same way, we also get that
    \begin{align*}
        \left(B_r^{p,k-p}\right)^\perp=\star Z_r^{Q-p,n-k-Q+p}
    \end{align*}
    and hence the claim follows.
\end{proof}
\begin{remark}
    Notice that we have shown that the collection of all complexes $\lbrace\left(E_{j,l}^{\bullet,\bullet},\Delta_j\right)\rbrace_{j\in I_{\bullet,\bullet}}$ is Hodge-$\star$ closed, but it this does not mean that each complex will be. More explicitly, given 
    \begin{align*}
        \alpha\in E_{r_1,r_2}^{p,k-p} \text{ then }\star\alpha\in E_{r_2,r_1}^{Q-p,n-k-Q+p}
    \end{align*}
    however the two subspaces $ E_{r_1,r_2}^{p,k-p}$ and $E_{r_2,r_1}^{Q-p,n-k-Q+p}$ may not belong to the same complex, i.e. there may not be joined by a string of non trivial maps $\Delta_j$. This becomes evident in the following explicit example.
\end{remark}
\subsection{An explicit example: the Engel group}\label{engel group}
Let us consider the filiform 4-dimensional nilpotent Lie group $\mathbb G$, usually referred to as the Engel group. This is the connected, simply-connected nilpotent Lie group whose Lie algebra $\mathfrak{g}$ has the following non trivial brackets
\begin{align*}
    [X_1,X_2]=X_3\ ,\ [X_1,X_3]=X_4\,.
\end{align*}
Since its Lie algebra admits a stratification \begin{align*}
    \mathfrak{g}=V_1\oplus V_2\oplus V_3=\mathrm{span}_{\mathbb R}\lbrace X_1,X_2\rbrace\oplus\mathrm{span}_{\mathbb R}\lbrace X_3\rbrace\oplus\mathrm{span}_{\mathbb R}\lbrace X_4\rbrace
\end{align*} this group can be endowed with a Carnot group structure, and so the de Rham complex is a multicomplex:
\begin{align*}
    \left(\Omega^\bullet(\mathbb G),d=d_0+d_1+d_2+d_3\right)\,.
\end{align*}
Explicit step-by-step computations of the Rumin complex $(E_0^\bullet,d_c)$ on the Engel group can be found in the literature \cite{rumin_grenoble,GagliardoNiremb,F+T1}. Here, we will only mention the properties of forms and of the Rumin complex that are strictly necessary to the construction of the spectral complexes $\lbrace (E_{j,l}^{\bullet,\bullet},\Delta_j)\rbrace_{j\in I_{\bullet,\bullet}}$.

Once we introduce a scalar product adapted to the stratification, we can assume $\lbrace X_1,X_2,X_3,X_4\rbrace$ to be an orthonormal basis of $\mathfrak{g}$ and construct its dual basis $\lbrace \theta_1,\theta_2,\theta_3,\theta_4\rbrace$ such that $\theta_i(X_j)=\delta_{ij}$. On Carnot groups, the concept of weights of forms is linked to the stratification, so that
\begin{align*}
    w(\theta_1)=w(\theta_2)=1\, ,\ w(\theta_3)=2&\,,\text{ and }w(\theta_4)=3\\
    w(\theta_1\wedge\theta_2)=2\,,\ w(\theta_1\wedge\theta_3)=w(\theta_2\wedge\theta_3)=3\,,\ &w(\theta_1\wedge\theta_4)=w(\theta_2\wedge\theta_4)=4\,,\ w(\theta_3\wedge\theta_4)=5 \\
    w(\theta_1\wedge\theta_2\wedge\theta_3)=4\,,\ w(\theta_1\wedge\theta_2\wedge\theta_4)&=5\,,\ w(\theta_1\wedge\theta_3\wedge\theta_4)=w(\theta_2\wedge\theta_3\wedge\theta_4)=6\\
    w(\theta_1\wedge&\theta_2\wedge\theta_3\wedge\theta_4)=7\,.
\end{align*}
Furthermore, after explicit computations, one can find
\begin{itemize}
    \item $E_0^0=C^\infty(\mathbb G)=\C_{0,0}$, here smooth functions are taken to be 0-forms of weight 0;
    \item $E_0^1=\mathrm{span}_{C^\infty(\mathbb G)}\lbrace\theta_1,\theta_2\rbrace=\C_{1,0}$;
    \item $E_0^2=\mathrm{span}_{C^\infty(\mathbb G)}\lbrace \theta_2\wedge\theta_3,\theta_1\wedge\theta_4\rbrace\subset\C_{3,-1}\oplus\C_{4,-2}$;
    \item $E_0^3=\mathrm{span}_{C^\infty(\mathbb G)}\lbrace \theta_1\wedge\theta_3\wedge\theta_4,\theta_2\wedge\theta_3\wedge\theta_4\rbrace=\C_{6,-3}$;
    \item $E_0^4=\mathrm{span}_{C^\infty(\mathbb G)}\lbrace \theta_1\wedge\theta_2\wedge\theta_3\wedge\theta_4\rbrace=\C_{7,-3}$;
\end{itemize}
which implies that
\begin{itemize}
    \item $d_c\colon E_0^0\longrightarrow E_0^1$ is given by $d_c=d_c^1$;
    \item $d_c\colon E_0^1\longrightarrow E_0^2$ is given by $d_c=d_c^2+d_c^3$;
    \item $d_c\colon E_0^2\longrightarrow E_0^3$ is given by $d_c=d_c^2+d_c^3$. However, when considering forms of different weight separately, the Rumin differential is of only one bidegree:\\ $d_c\colon E_0^2\cap\C_{3,-1}\longrightarrow E_0^3$ is given by $d_c=d_c^3$, and \\$d_c\colon E_0^2\cap\C_{4,-2}\longrightarrow E_0^3$ is given by $d_c=d_c^2$;
    \item $d_c\colon E_0^3\longrightarrow E_0^4$ is given by $d_c=d_c^1$.
\end{itemize}
In other words, there is only one choice of bidegree $(p,k)$ for which $I_{p,k}$ contains more than 1 element, namely $(1,1)$ with $I_{1,1}=\lbrace 2,3\rbrace$. Consequently, the corresponding collection $\lbrace(E_{j,l}^{\bullet,\bullet},\Delta_j)\rbrace_{j\in I_{\bullet,\bullet}}$ will feature only two complexes. This can be visualised using a diagram.

Just a quick warning for people who are used to dealing with spectral sequences: we are using a pictorial convention often employed when dealing with differential forms in Carnot groups. This means that the spaces $\C_{p,k-p}$ are not placed at the point $(p,k-p)$ of the Cartesian plane. Instead, the weight $p$ is encoded along
the vertical axis, and the degree $p+k-p=k$ along the horizontal axis, so that each space $\C_{p,k-p}$ will be placed at the point $(Q-p,k)$ of the Cartesian plane, where $Q$ is the weight of the volume form. The naive idea is that, when dealing with the de Rham complex $(\Omega^\bullet(\mathbb G),d)$ on Carnot groups, we start from the top left corner with 0-forms of weight 0. Moving to the right, the degree of the forms increases, whereas going downwards the weight of the forms increases.

Using this convention, the diagram below is depicting the Rumin complex on the Engel group, where only the non trivial spaces $E_0^\bullet$ are included.
Moreover, the thick lines indicate the operators $d_c^j$ which are non-zero when acting on the spaces of Rumin forms. 

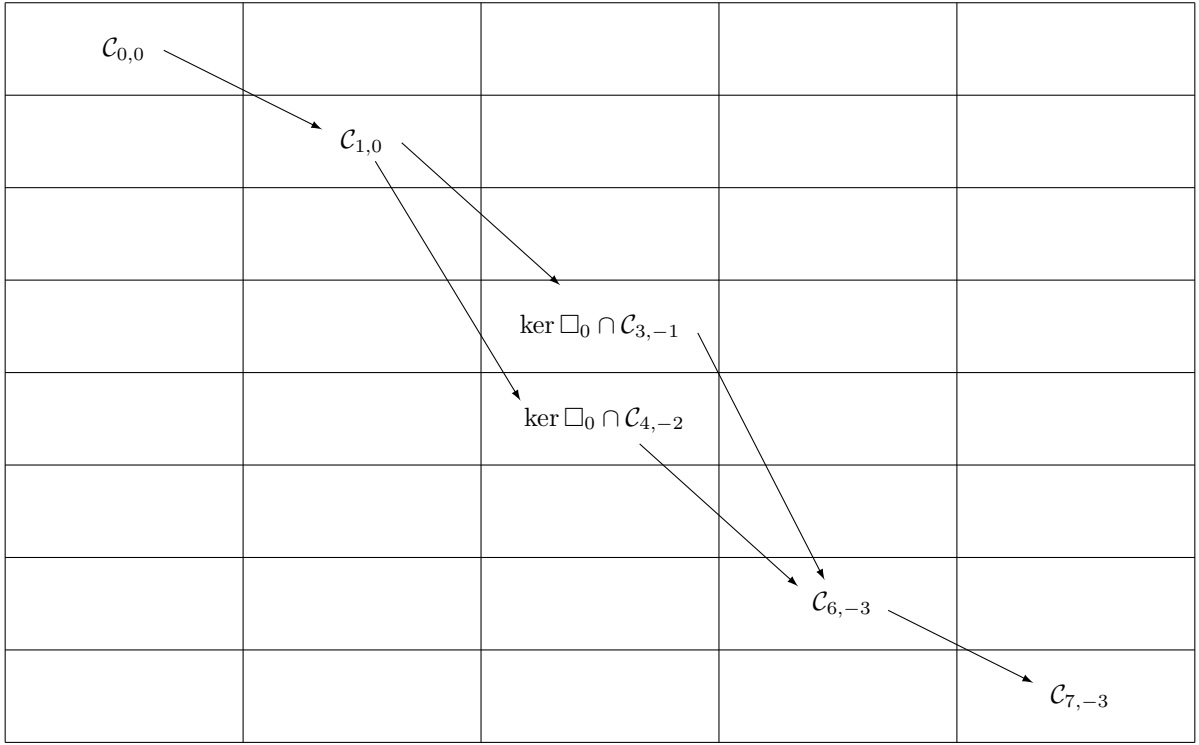
\begin{align*}
    \begin{picture}(450,280)(0,0)
   \put(  0,  0){\line( 1, 0){450}}
   \put(  0,  0){\line( 0, 1){280}}
   \multiput(  0, 35)( 0,35){8}{\line( 1, 0){450}}
   \multiput( 90,  0)(90, 0){5}{\line( 0, 1){280}}
   \put(45,262){\makebox(0,0){$ \C_{0,0}$}}
   \put(135,227){\makebox(0,0){$ \C_{1,0}$}}
   \put(225,157){\makebox(0,0){$\ker\Box_0\cap\C_{3,-1}$}}
   \put(225,122){\makebox(0,0){{ $\ker\Box_0\cap\C_{4,-2}$}}}
   \put(315,52){\makebox(0,0){{ $\C_{6,-3}$}}}
   \put(405,17){\makebox(0,0){{ $\C_{7,-3}$}}}
   \put(60,262){\vector(1,-0.5){60}}
   \put(150,227){\vector(1,-0.9){60}}
   \put(140,220){\vector(1,-1.65){55}}
   \put(262,155){\vector(1,-1.95){48}}
   \put(240,113){\vector(1,-0.9){60}}
   \put(334,50){\vector(1,-0.5){55}}
  \end{picture} 
\end{align*}
As already mentioned, we focus on bidegree $(1,1)$ to see that $I_{1,1}=\lbrace 2,3\rbrace$, so that we are going to construct the following two subcomplexes
\begin{align*}
    Z_1^{0,0}\xrightarrow[]{\Delta_1}Z_2^{1,0}\cap\left(B_1^{1,0}\right)^\perp\xrightarrow[]{\Delta_2}Z_3^{3,-1}\cap\left(B_2^{3,-1}\right)^\perp\xrightarrow[]{\Delta_3}Z_1^{6,-3}\cap\left(B_3^{6,-3}\right)^\perp\xrightarrow[]{\Delta_1} \C_{7,-3}\cap\left(B_1^{7,-3}\right)^\perp\\
    Z_1^{0,0}\xrightarrow[]{\Delta_1}Z_3^{1,0}\cap\left(B_1^{1,0}\right)^\perp\xrightarrow[]{\Delta_3}Z_2^{4,-2}\cap\left(B_3^{4,-2}\right)^\perp\xrightarrow[]{\Delta_2} Z_1^{6,-3}\cap\left(B_2^{6,-3}\right)^\perp\xrightarrow[]{\Delta_1}\C_{7,-3}\cap\left(B_1^{7,-3}\right)^\perp
\end{align*}
In this particular case, the expression of the spaces $E_{j,l}^{\bullet,\bullet}$ simplifies greatly:
\begin{itemize}
    \item $Z_1^{0,0}=\ker\Box_0\cap\C_{0,0}=E_0^0$;
    \item $Z_2^{1,0}=Z_1^{1,0}$ and so $E_{2,1}^{1,0}=Z_1^{1,0}\cap\left(B_1^{1,0}\right)^\perp=\ker\Box_0\cap\C_{1,0}=E_0^1$;
    \item $Z_3^{1,0}=\lbrace \alpha\in \C_{1,0}\mid d_c^2\overline{\alpha}=0\rbrace$ and so $E_{3,1}^{1,0}=\lbrace \overline{\alpha}\in E_0^1\mid d_c^2\overline{\alpha}=0\rbrace$;
    \item $B_2^{3,-1}=B_1^{3,-1}$ and $Z_3^{3,-1}=Z_1^{3,-1}$ and so $E_{3,2}^{3,-1}=\ker\Box_0\cap\C_{3,-1}=E_0^2\cap \C_{3,-1}$;
    \item $B_3^{4,-2}=B_1^{4,-2}$ and $Z_2^{4,-2}=Z_1^{4,-2}$ and so $E_{2,3}^{4,-2}=\ker\Box_0\cap\C_{4,-2}=E_0^2\cap\C_{4,-2}$;
    \item $B_3^{6,-3}=\lbrace \overline{\alpha}\in\ker\Box_0\cap\C_{6,-3}\mid \overline{\alpha}=d_c^2\overline{c}_{4}\text{ for some }\overline{c}_4\in\ker\Box_0\cap\C_{4,-2}\rbrace$ and so $E_{1,3}^{6,-3}=\ker\Box_0\cap\left(\mathrm{Im}\,d_c^2\big\vert_4\right)^{\perp}=E_0^3\cap\left(\mathrm{Im}\,d_c^2\big\vert_4\right)^\perp$;
    \item $B_2^{6,-3}=B_1^{6,-3}$ and so $E_{1,2}^{6,-3}=\ker\Box_0\cap \C_{6,-3}=E_0^3$;
    \item $\C_{7,-3}\cap\left(B_1^{7,-3}\right)^\perp=\ker\Box_0\cap\C_{7,-3}=E_0^4$.
\end{itemize}
Moreover, each one of the differentials $\Delta_j$ coincides with the corresponding operators $d_c^j$, and so we are left with the following two complexes
\begin{align*}
    E_0^0\xrightarrow[]{d_c^1}&\ \quad \ \ E_0^1\ \ \ \ \xrightarrow[]{d_c^2}E_0^2\cap\C_{3,-1}\xrightarrow[]{d_c^3}E_0^3\cap\left(\mathrm{Im}\,d_c^2\right)^\perp\xrightarrow[]{d_c^1} E_0^4\\
    E_0^0\xrightarrow[]{d_c^1}&E_0^1\cap\ker d_c^2\xrightarrow[]{d_c^3} E_0^2\cap\C_{4,-2}\xrightarrow[]{d_c^2} \qquad\ E_0^3\qquad\ \xrightarrow[]{d_c^1} E_0^4
\end{align*}
Let us study the cohomology spaces that these complexes compute. For the first one, we have
\begin{itemize}
    \item $\ker d_c^1\colon E_0^0\to E_0^1=Z_2^{0,0}=Z_\infty^{0,0}$;
    \item $\ker d_c^2\colon E_0^1\to E_0^2\cap\C_{3,-1}=Z_3^{1,0} \subsetneq Z_4^{1,0}   $ and $\mathrm{Im}\,d_c^1\colon E_0^0\to E_0^1=B_2^{1,0}=B_\infty^{1,0}$;
    \item $\ker d_c^3\colon E_0^2\cap\C_{3,-1}\to E_0^3\cap\left(\mathrm{Im}\,d_c^2\right)^\perp=Z_4^{3,-1}=Z_\infty^{3,-1}$ and $\mathrm{Im}\,d_c^2\colon E_0^1\to E_0^2\cap\C_{3,-1}=B_3^{3,-1}=B_\infty^{3,-1}$;
    \item $\ker d_c^1\colon E_0^3\cap\left(\mathrm{Im}\,d_c^2\right)^\perp\to E_0^4=Z_2^{6,-3}=Z_\infty^{6,-3}$ and $\mathrm{Im}
    \,d_c^3\colon E_0^3\cap\C_{3,-1}\to E_0^3\cap\left(\mathrm{Im}\,d_c^2\right)^\perp=B_4^{6,-3}=B_\infty^{6,-3}$;
    \item $\mathrm{Im}\,d_c^1\colon E_0^3\cap\left(\mathrm{Im}\,d_c^2\right)^\perp\to E_0^4=B_2^{7,-3}=B_\infty^{7,-3}$
\end{itemize}
while for the second
\begin{itemize}
    \item  $\ker d_c^1\colon E_0^0\to E_0^1\cap\ker d_c^2=Z_2^{0,0}=Z_\infty^{0,0}$;
    \item $\ker d_c^3\colon E_0^1\cap\ker d_c^2\to E_0^2\cap\C_{4,-2}=Z_4^{1,0} = Z_\infty^{1,0}   $ and $\mathrm{Im}\,d_c^1\colon E_0^0\to E_0^1\cap\ker d_c^2=B_2^{1,0}=B_\infty^{1,0}$;
    \item $\ker d_c^2\colon E_0^2\cap\C_{4,-2}\to E_0^3=Z_3^{3,-1}=Z_\infty^{3,-1}$ and $\mathrm{Im}\,d_c^3\colon E_0^1\cap\ker d_c^2\to E_0^2\cap\C_{4,-2}=B_3^{4,-2}=B_\infty^{4,-2}$;
    \item $\ker d_c^1\colon E_0^3\to E_0^4=Z_2^{6,-3}=Z_\infty^{6,-3}$ and $\mathrm{Im}
    \,d_c^2\colon E_0^3\cap\C_{3,-1}\to E_0^3=B_3^{6,-3}\subsetneq B_4^{6,-3}$;
    \item $\mathrm{Im}\,d_c^1\colon E_0^3\to E_0^4=B_2^{7,-3}=B_\infty^{7,-3}$.
\end{itemize}
Therefore, the cohomology of the original multicomplex (the de Rham cohomology in this case) is carried:
\begin{itemize}
    \item in degree 0, 2 and 4 by both subcomplexes;
    \item in degree 1 by the second one;
    \item in degree 3 by the first one.
\end{itemize}
This follows directly from the theory of spectral sequences: indeed in degree 0, 2, and 4 the whole cohomology is the direct sum of the cohomology groups of each subcomplex. In degree 0 and 4, we are simply taking the sum of the same group, whereas in degree 2 the sum is not trivial.
\bigskip
		\bibliographystyle{amsplain}

\bibliography{bibliography}
\end{document}